\documentclass[11pt,twoside]{amsart}
\usepackage{shuffle}
\usepackage{lmodern}
\usepackage{mathbbol}
\usepackage{adjustbox}
\usepackage{bm}
\usepackage[usenames,dvipsnames]{xcolor}
\usepackage[colorlinks=true,linkcolor=NavyBlue,urlcolor=RoyalBlue,citecolor=PineGreen,
	hypertexnames=false]{hyperref}
\makeatletter
\renewcommand{\tocsection}[3]{%
	\indentlabel{\@ifnotempty{#2}{\bfseries\ignorespaces#1 #2.\,\,}}\bfseries#3}
\renewcommand{\tocsubsection}[3]{%
	\indentlabel{\@ifnotempty{#2}{\ignorespaces#1 #2\quad}}#3}
\renewcommand{\tocsubsubsection}[3]{%
	\quad\quad\quad\indentlabel{\@ifnotempty{#2}{\ignorespaces#1 #2\quad}}#3}

\newcommand\@dotsep{4.5}
\def\@tocline#1#2#3#4#5#6#7{\relax
	\ifnum #1>\c@tocdepth 
	\else
		\par \addpenalty\@secpenalty\addvspace{#2}%
		\begingroup \hyphenpenalty\@M
		\@ifempty{#4}{%
			\@tempdima\csname r@tocindent\number#1\endcsname\relax
		}{%
			\@tempdima#4\relax
		}%
		\parindent\z@ \leftskip#3\relax \advance\leftskip\@tempdima\relax
		\rightskip\@pnumwidth plus1em \parfillskip-\@pnumwidth
		#5\leavevmode\hskip-\@tempdima{#6}\nobreak
		\leaders\hbox{$\m@th\mkern \@dotsep mu\hbox{.}\mkern \@dotsep mu$}\hfill
		\nobreak
		\hbox to\@pnumwidth{\@tocpagenum{\ifnum#1=1\bfseries\fi#7}}\par
		\nobreak
		\endgroup
	\fi}
\AtBeginDocument{%
	\expandafter\renewcommand\csname r@tocindent0\endcsname{0pt}
}
\newcommand\blfootnote[1]{%
	\begingroup
	\renewcommand\thefootnote{}\footnote{#1}%
	\addtocounter{footnote}{-1}%
	\endgroup
}
\def\l@subsection{\@tocline{2}{0pt}{2.5pc}{5pc}{}}
\makeatother
\usepackage{tikz}
\usepackage{tikz-cd} \usepackage{subfigure}
\usepackage{verbatim}
\usepackage[margin=0.7in]{geometry}

\usepackage[T1]{fontenc}
\usepackage{lmodern}

\usepackage{amsthm}
\usepackage{amsmath}
\usepackage{amssymb}
\usepackage{amsfonts}
\usepackage{enumerate}
\usepackage[english]{babel}
\usepackage{mathtools}
\usepackage{stmaryrd}
\usepackage[font=small]{caption}
\DeclareFontFamily{U}{MnSymbolC}{}
\DeclareFontShape{U}{MnSymbolC}{m}{n}{
	<-5.5> MnSymbolC5
	<5.5-6.5> MnSymbolC6
	<6.5-7.5> MnSymbolC7
	<7.5-8.5> MnSymbolC8
	<8.5-9.5> MnSymbolC9
	<9.5-11.5> MnSymbolC10
	<11.5-> MnSymbolCb12
}{}

\newcommand{\lmss}[1]{\textrm{\normalfont{{\fontfamily{lmss}\selectfont #1}}}}

\newcommand{\sbt}{\,\begin{tikzpicture}[baseline=(X.base)]%
		\node[draw, fill,black,circle, inner sep=1pt] at (0,0.1) {};
		\node[circle,inner sep=0pt,outer sep=0pt] (X){$\ $};
	\end{tikzpicture}%
	\,
}
\newtheorem{theorem}{Theorem}

\theoremstyle{plain}
\newtheorem{proposition}[theorem]{Proposition}
\newtheorem*{proposition*}{Proposition}

\newtheorem*{theorem*}{Theorem}

\newtheorem{lemma}[theorem]{Lemma}
\theoremstyle{definition}
\newtheorem{remarque}{Remark}
\newtheorem{definition}[theorem]{Definition}

\def\gtimes{{\color{OliveGreen} \, \pmb{\otimes} \,}}
\def\oboxtimes{{\color{orange} \,\boxtimes\,}}

\newcommand{\boxdtimes}{\raisebox{-1pt}{$\oboxtimes$}}

\address{Department of Mathematical Sciences, Norwegian University of Science and Technology (NTNU), 7491 Trondheim, Norway.}

\title[Shuffle algebra and operator-valued probability theory]{A shuffle algebra point of view on\\ operator-valued probability theory}
\author{Nicolas Gilliers}

\begin{document}

\begin{abstract}
  We extend the shuffle algebra perspective on scalar-valued non-commutative probability theory to the operator-valued case.
	Given an operator-valued probability space with an algebra $B$ acting on it (on the left and on the right), we associate operators in the operad of multilinear maps on $B$ to the operator-valued distribution and free cumulants of a random variable. These operators define a representation of a PROS of non-crossing partitions. Using concepts from higher category theory, specifically $2$-monoidal categories, we define a notion of unshuffle Hopf algebra on an underlying PROS.
	We introduce a PROS of words insertions and show that both the latter and the PROS of non-crossing partitions are unshuffle Hopf algebras. The two relate by mean of a map of unshuffle bialgebra (in a $2$-monoidal sense) which we call the splitting map.
  Ultimately, we obtain a left half-shuffle fixed point equation corresponding to free moment-cumulant relations in a shuffle algebra of bicollection homomorphisms on the PROS of words insertions. Right half-shuffle and shuffle laws are interpreted in the framework of boolean and monotone non-commutative probability theory, respectively.
\end{abstract}
\keywords{operator-valued non-commutative probability theory, higher category theory, duoidal categories, operads, properads, PROS, shuffle algebra, half-shuffles}
\subjclass{46L53, 46L54, 18M60, 18M65, 18M80,16W25}
\maketitle
\blfootnote{Date: \today}
\tableofcontents

\section{Introduction}
\label{sec:intro}
The main objective of the present work is to extend the shuffle algebraic perspective on free, boolean and monotone moment-cumulant relations to the setting of operator-valued probability theory.
\subsection{Motivation and overview}
In classical probability theory, it is now well established that moment-cumulant relations are best understood in the context of \emph{M\"obius inversion} on the lattice of set partitions and its associated \emph{incidence co-algebra}, see \cite{speed1983cumulants}. The combinatorial side of Voiculescu's (scalar-valued) free probability theory finds its roots in the seminal work of Speicher \cite{speicher1994multiplicative}, who developed Rota's work by showing that upon replacing set partitions by \emph{non-crossing set partitions}, M\"obius inversion could be used to define an equivalent notion of cumulant in free probability. More precisely, in free probability theory, moments and cumulants are seen as linear maps on the incidence coalgebra of the lattice of non-crossing partitions and the free moment-cumulant relations are expressed in terms of the convolution product of the cumulant map with the zeta function. We refer the reader to \cite{mingo2017free,nica2006lectures} for an introduction to the theory of free probability.

When considering operator-valued moments and cumulants, Speicher's results can be (partially) extended \cite{speicher1998combinatorial}. Let $(\mathcal{A}, e, B)$ be an operator-valued probability space. Recall that $B$ is an algebra acting on the right and  on the left on the involutive algebra $\mathcal{A}$ and $e : \mathcal{A} \rightarrow B$ is a $B$-$B$ linear map \cite{mingo2017free}. The operator-valued expectation $e$ is extended to the lattice $\textrm{NC}$ of non-crossing partitions as a \emph{multiplicative function} $E=(e_{\pi})_{\pi\in\textrm{NC}}$. One contribution of the present work is to give a precise meaning to this multiplicativity property by using operads. For the time being, multiplicativity refers to the fact that $E(\pi)$ can be computed by \emph{composing}, in a certain sense, the values of $E$ on each block of $\pi$. In comparison to the scalar-valued case, $E(\pi)$ does also depend on the \emph{nesting} of the blocks in the non-crossing partition $\pi$ (two blocks of a partition are nested if one is contained in the convex closure of the other). Altogether operator-valued free cumulants define a function $K$, also multiplicative, on $NC$ and depend as well on the nesting of the blocks. The convolution of $E$ and $K$ with a \emph{scalar-valued} function makes sense, giving rise to operator-valued moment-cumulant relations. Extracting algebraic structures encoding the nesting of blocks is then primordial to a better understanding of the properties of free cumulants. As explained in this work, it also participates in a concise description of relations with their boolean and monotone counterparts.

Recently, K. Ebrahimi-Fard and F. Patras proposed a rather different perspective on moment-cumulant relations in the scalar-valued case \cite{ebrahimi2015cumulants,ebrahimi2016splitting,ebrahimi2018monotone}.

Their point-of-view does not involve M\"obius inversion on lattices of set partitions. Instead, it is based on \emph{combinatorial Hopf algebras}. More precisely, by describing a genuine shuffle algebra on words, \emph{(cumulants) moments} are encoded as values taken by some \emph{Hopf algebra (infinitesimal) characters}. This setting allows for a unified picture of the three different types of cumulants in non-commutative probability, i.e., free, monotone and boolean, as three faces of a single object, the unshuffle coproduct. This approach naturally gives rise to a (pre-)Lie theoretic description of the relations between the different cumulants in terms of shuffle adjoint transformations. It is critical to notice that the shuffle algebra setting does not involve at any point non-crossing partitions and that it has recently been successfully applied in the context of infinitesimal probability, provided that the base field of complex numbers is replaced by the Grassmann algebra, see \cite{celestino2019relations}. In the case of present interest, the target algebra of the morphisms we consider is non-commutative. As a result, the (pre-)Lie theoretic machinery developed in \cite{ebrahimi2018monotone} fails to work in the context of operator-valued probability spaces.

Until recently, it was unclear how the two perspectives, i.e., M\"obius inversion on the lattice of non-crossing set partitions on one hand and shuffle algebra on words on the other, could be related. In \cite{ebrahimi2019operads}, the authors started to address this question. They showed that lattice and shuffle algebra approaches are governed each by their respective operad of non-crossing partitions and the associated incidence co-algebras. The shuffle algebra approach is associated with the so-called \emph{gap-insertion operad of non-crossing partitions}, which is going to be extensively used in this work, while the M\"obius inversion formulation is encoded by the incidence coalgebra of a partition-refinement operad. The incidence bialgebra of the gap-insertion operad bears an unshuffle algebraic structure.

In the shuffle approach, the two functions $E$ and $K$ above are extended as algebra morphisms on the incidence bialgebra of the gap-insertion operad, solutions of the the following fixed point equations:
\begin{equation*}
	E = \varepsilon + e \prec E, \quad K = \varepsilon + k \prec K
\end{equation*}
The two infinitesimal morphisms $e$ and $k$ encode moments and free cumulants of all orders.

Our approach to extend the shuffle algebraic perspective on free, boolean and monotone moment-cumulant relations to the setting of operator-valued probability theory relies on the first part of reference \cite{ebrahimi2019operads}. We explain how considering moments and free cumulants of an operator-valued probability space as multiplicative functions on the lattice of non-crossing partitions naturally leads to an operadic perspective. Such a point of view encompasses the boolean cumulants as well seen as ``almost'' operadic morphisms on the word insertion operad. This result extends the picture developed in \cite{ebrahimi2019operads} to the operator-valued case.

As already mentioned, in this context both the \emph{nesting and the linear ordering of the blocks} of a non-crossing partition are essential and are algebraically implemented into the gap-insertion operad. The multiplicativity property of the moment map $E$ leads to an interpretation of the latter as values taken by an operadic morphism on the gap-insertion operad.
Since the moments associated with each block of a partition do not commute with each other (even in the case of a single random variable) polynomials in the incidence coalgebra of the gap-insertion operad should be considered as operators with multiple outputs. Incidentally, the bialgebraic structure should be replaced by a co-PROS/PROS structure. To be more precise, a word build from non-crossing set partitions (including the partition of the empty set) is an operator with as many inputs as gaps between the elements of the partitioned sets. A single output is associated with each partition in the word. The co-PROS structure (which is actually a simpler version of a plain non-symmetric coproperadic structure) is then (gradued) dual to the gap-insertion operad. A word on non-crossing set partitions should be seen as ``a horizontal object'' and applying the coproduct map on such a word results in two words that are vertically stacked.

We show that this new insight finds a transparent description by means of a so-called duoidal structure on bicollections (graded vector spaces with two gradings standing for the number of inputs and the outputs of an operator). A \emph{duoidal category} is endowed with two tensor products (we use the symbols $\raisebox{-1pt}{$\oboxtimes$}$ and $\gtimes$ throughout the article) satisfying a Lax property.
We shall use the terminology \emph{vertically}, respectively \emph{horizontally}, for sub-categories of the category of bicollections, or objects related to the monoidal structure \raisebox{-1pt}{$\oboxtimes$}, respectively $\gtimes$.
After having expounded the duoidal structure of the category of bicollections, we proceed to define the equivalent notion of a $\raisebox{-1pt}{$\oboxtimes$} \gtimes$-Hopf algebra. The latter has both a vertical product and vertical coproduct which are compatible through a horizontal algebraic structure.

Free and boolean cumulants are implemented as (horizontal) algebra morphisms for the concatenation product on the space of words on non-crossing partitions. In the free case, this morphism is also a PROS morphism. However, this does not hold in the boolean case and we obtain a morphism with a multiplicative property reminiscent to properties satisfied by morphisms on an operad of words-insertions. The convolution monoid of horizontal algebra morphisms on a \raisebox{-1pt}{$\oboxtimes$}$\gtimes$-Hopf algebra valued in a PROS of endomorphisms provides a unifying description of the gap-insertion and (almost) word insertion PROS morphisms.

We enrich the structure of \raisebox{-1pt}{$\oboxtimes$}$\gtimes$-Hopf algebra by introducing the notion of unshuffle Hopf algebra in a duoidal category. Once again, we show that the $\raisebox{-1pt}{$\oboxtimes$} \gtimes$-Hopf algebra of words on non-crossing partitions can be endowed with such a structure. The dual of this unshuffle structure gives rise to a shuffle algebraic structure on the class of bicollection morphisms from the PROS of words on non-crossing partitions to the PROS of multilinear maps on $B$. The operator-valued moments and free cumulants implemented as operadic morphisms on the gap-insertion operad satisfy, separately, left half-shuffle fixed point equations. The horizontal morphism implementing boolean cumulants is the solution of a right half-shuffle fixed point equation.

The introduction of a second monoidal structure is supported by the fact that the (pre-)Lie theoretic perspective on scalar moments and free, boolean, monotone cumulants is fully extended to the operator-valued case. In particular, the notion of infinitesimal character makes sense in this setting and requires horizontal composition of partitions (words), while the vertical direction (operadic composition of non-crossing partitions) is used to define the monoid the operator-valued moments, free and boolean cumulants are elements of.

The free and boolean moment-cumulant relations are then retrieved as fixed point equations in a shuffle algebra of bicollection morphisms on a PROS of words on random variables. This second shuffle algebra relates to the one associated to non-crossing partitions by mean of a shuffle algebra morphism, the so-called \emph{splitting map}. The half-shuffle fixed point equations are obtained as pulling-backs of half-shuffle fixed point equations satisfied by the boolean and free cumulants.

In the context of non-commutative probability theory, various authors have used Hopf algebras and operads from different perspectives. We mention the work of Friedrich--McKay \cite{friedrich2015homogeneous}, Hasebe--Lehner, \cite{hasebe2017cumulants}, Mastnak--Nica \cite{mastnak2010hopf} as well as the work of Gabriel \cite{gabriel2015combinatorial}. In the latter, the author defines Hopf algebraic structures related to additive and multiplicative convolutions by using a geometric perspective on the space of (non-crossing) partitions. Operadic approaches to moment-cumulant relations have already been exploited by Joshuat-Verg\`es, Menous, Thibon and Novelli in \cite{josuat2017free} to obtain an operadic version of the shuffle point of view developed by Ebrahimi-Fard and Patras. Another perspective on moment-cumulant relations in an operadic framework was developed by Drummond-Cole in \cite{drummond2016operadic} and \cite{Drummond-Cole2018}.
We end our (non-exhaustive) summary about previous works related to operator-valued probability theory with the two papers \cite{dykema2006stransform,dykema2007multilinear} by Dykema, together with the following remark. In these two papers the point of view adopted by the author is fundamentally analytical. This translates in the way non-crossing partitions are considered as operators. It is radically different from our approach. For instance, in Dykema's work a partition of a set $S$ has $|S|-1$ inputs, while in our case such a partition has $|S|+1$ inputs.

\medskip

{\bf{Acknowledgements}}: The author would like to thank Kurusch Ebrahimi-Fard and Joachim Kock for fruitful discussions. Nicolas Gilliers is supported by the ERCIM Alain Bensoussan fellowship programme.
\subsection{Operator-valued probability theory}

We start with a small (historical) account on free probability theory and its operator-valued version. Free probability theory was created in 1985 by Dan Voiculescu to understand free factors of von Neumann algebras. Originally developed in the vicinity of the theory of algebras of operators, \emph{freeness} drew probabilists' attention as the right algebraic framework to compute the asymptotic distribution of large random matrices. Creating a common notion encompassing (finite dimensional) distribution of random matrices and their asymptotic requires a further step in the abstraction, notably about what we understand as a \emph{probability space}. In a nutshell, a probability space allows for taking sums, products of random variables (its elements) and compute \emph{moments} of the latter. This last requirement implies that it is endowed with a linear map enjoying a notion of \emph{positivity} and called \emph{state}.

The very first example of a probability space is the commutative algebra of essentially bounded random variables $L^{\infty}(\Omega,\mathcal{F},\mathbb{P})$ on a classical probability space endowed with the usual expectation. In general, a probability space is a possibly non-commutative von Neumann algebra.

In classical probability, the \emph{conditional expectation} is a map acting on a space of essentially bounded random variables measurable with respect to a $\sigma$-field $\mathcal{F}_{1}$ valued in a smaller algebra of random variables measurable with respect to a sub $\sigma$-field $\mathcal{F}_{2} \subset \mathcal{F}_{1}$.

As such, the conditional mean of a random variable with respect to the sigma field $\mathcal{F}_{2}$ is not scalar-valued but algebra valued. Still, it enjoys the same positivity property as the scalar expectation does. Besides, it is linear with respect to left and right multiplication by random variables measurable with respect to the smaller sigma field.

These properties are algebraically translated in the settings of non-commutative probability as follows.

An operator-valued probability space $(\mathcal{A},E,B)$ is a bi-module involutive complex unital algebra $\mathcal{A}$ over an unital involutive algebra $B$ together with a $B$-bimodule positive unital morphism, $E: \mathcal{A}\rightarrow B$. In symbols, with $a\in \mathcal{A},~b_{1},b_{2} \in B$,
$$(b_{1}\cdot a)\cdot b_{2} = b_{1} \cdot (a \cdot b_{2}),\, (ba)^{\star} = a^{\star} b^{\star}, E(b_{1}ab_{2})=b_{1}E(a)b_{2},\, E(aa^{\star}) \in BB^{\star}, E(1_{\mathcal{A}})=1_{B}$$
We define boolean, free and monotone conditional cumulants using M\"obius inversion. As for the scalar-valued case, conditional free, boolean and monotone independence is characterized by the vanishing of mixed cumulants. The reader is directed to the monograph \cite{speicher1998combinatorial} for a detailed introduction on the combinatorial aspect of operator-valued probability theory. We denote by NC($n$) the set of all non-crossing partitions of $\llbracket 1,n \rrbracket$ and by $1_{n}$ the unique partition of NC($n$) with only one block.

For simplicity, we pick a single random variable $a \in \mathcal{A}$. The $B$-valued distribution of the random variable $a$ is the collection of elements in $B$:
\begin{equation}
	\label{distribution}
	E(b_{0}ab_{1}a\cdots ab_{n}), \quad b_{0},\ldots,b_{n}\in B,~n \geq 1.
\end{equation}
Let $a \in \mathcal{A}$ be a random variable, we denote by $B\!\left[a\right]$ the smallest $B$-$B$ bi-module algebra containing $a$. Speicher's original recursive definition of $e_{\pi}:B\!\left[a\right]^{\otimes_{B} |\pi|} \rightarrow B,~\pi \in \mathrm{NC}$ is as follows:
\begin{align}
	 & e_{\mathbf{1}_{n}}(a_{1}\otimes \cdots \otimes a_{n}) = E(a_{1}\cdots a_{n+1}),   \label{def:epi}                                                                                              \\
	 & e_{\pi}(a_{1}\otimes\cdots\otimes a_{n}) = e_{\pi_{1}}\big(a_{1}\otimes \cdots \otimes a_{k-1} \otimes E(a_{k}\otimes\cdots\otimes a_{l})a_{l+1}\otimes\cdots\otimes a_{n}\big) \nonumber .
\end{align}
Here $\llbracket k,l \rrbracket$ is an interval in $\pi$ and $\pi_{1}$ is the restriction of $\pi$ to $\llbracket 1,n\rrbracket \backslash \llbracket k,l\rrbracket$. From this perspective, $e_{\pi}$ is a map that takes random variables as inputs, which can be pictured as sitting on the legs of the partition $\pi$.

The perspective developed in this work starts with a different point of view on the distribution of the random variable $a$. In fact, we see it as a collection of homomorphisms in the operad $\mathrm{Hom}(B)$ of multilinear maps on $B$,
\begin{equation*}
	\mathrm{Hom}(B)(n) = \mathrm{Hom}(B^{\otimes n},B)
\end{equation*}
and with $\alpha,\beta_{1},\ldots,\beta_{|\alpha|} \in \mathrm{Hom}(B)$,
\begin{equation*}
	\Big(\alpha \circ (\beta_{1},\ldots,\beta_{|\alpha|})\Big)(b_{1},\ldots,b_{|\alpha|}) = \alpha(\beta_{1}(b_{1},\ldots,b_{|\beta_{1}|}),\ldots,\beta_{|\alpha|}(b_{|\beta_{1}|+\cdots+|\beta_{|\alpha|-1}|},\ldots,b_{|\beta_{1}|+\cdots+|\beta_{|\alpha|}|}))
\end{equation*}
and
\begin{equation}
	E(a^{\otimes n}) \in \mathrm{Hom}(B^{\otimes\,(n+1)},B),~E(a^{\otimes n})(b_{0},\ldots,b_{n}) = E(b_{0}a b_{1}ab_{2} \cdots ab_{n}),~b_{0},\ldots,b_{n} \in B.
\end{equation}
We prove in a forthcoming section that the sub-operad of $\mathrm{Hom}(B)$ generated by the operators $E(a^{\otimes n})$ $~n\geq 1$ is a representation of the gap-insertion operad.
Denote by Int($n$) the set of all interval partitions in NC($n$). Let us recall the free and boolean moment-cumulant relations for operator-valued cumulants:
\begin{equation}
	\label{eqn:momentcumulants1}
	\tag{MC}
	E(a_{1}\cdots a_{n})
	= \sum_{\pi \in \mathrm{NC}(n)} \kappa_{\pi}(a_{1},\ldots,a_{n})
	= \sum_{I \in \textrm{Int}(n)} \beta_{I}(a_{1},\ldots,a_{n}),~a_{1},\ldots,a_{2} \in \mathcal{A}.
\end{equation}
In the last equations, the definitions of $\kappa_{\pi}$ and $\beta_{I}$ follow from equations \eqref{def:epi} with $\kappa_{1_{l-k+1}}(a_{k}\otimes \ldots\otimes a_{l})$ (respectively $\beta_{1_{l-k+1}}(a_{k} \otimes a_{l}))$ in place of $E(a_{k}\otimes\cdots\otimes a_{l})$.

Since $\kappa_{1_{n}}$ does not enter in the definition of $\kappa_{\pi}$ with $\pi \neq 1_{n}$, the first relation in \eqref{eqn:momentcumulants1} yields an inductive definition of the maps $\kappa_{1_{n}},~n\geq 1$:
\begin{equation*}
\kappa_{1_{n}}(a_{1}\otimes\cdots\otimes a_{n}) = E(a_{1}\cdots a_{n}) - \sum_{\substack{\pi \in \mathrm{NC}(n) \\ \pi \neq 1_{n}}} \kappa_{\pi}(a_{1}\otimes \cdots \otimes a_{n})
\end{equation*}
The inductive definition of the boolean cumulants $\beta_{1_{n}},~n\geq 1$ proceeds from:
\begin{equation*}
	\beta_{1_{n}}(a_{1},\ldots,a_{n}) = E(a_{1}\cdots a_{n}) - \sum_{\substack{I \in \mathrm{Int}(n) \\  I \neq 1_{n}}} \beta_{I}(a_{1}\ldots a_{n}).
\end{equation*}
\subsection{Outline}
We now outline the details of the relations \eqref{eqn:momentcumulants1}. First, we construct the operad $\mathcal{N}\mathcal{C}$ of non-crossing partitions in Section \ref{sec:gapinsertion}. We then construct operadic morphisms $E$ and $K$ from this operad to the operad of homomorphisms on $B$ implementing the set of cumulants $\kappa_{\pi}$ and moments $E_{\pi},~ \pi \in \mathcal{NC}$. We then address the problem of constructing a (convolution) monoid containing those two morphisms. To that aim, we introduce in Section \ref{sec:duoidal} the notion of duoidal category as well as a notion of Hopf algebra in this context. In a duoidal category, objects can {be} composed in two different --but compatible-- ways, either \emph{horizontally}, either \emph{vertically}. We can then define two categories of algebras, respectively two categories co-algebras, one for each tensor product. All of this is explained in Section \ref{sec:duoidal}.

The central result in Section \ref{sec:duoidal} is Lemma \ref{lem:lax}.
We show in Proposition \ref{prop:ncunshuffle} that the space of non-commutative polynomials on non-crossing partitions can be endowed with such a structure.  As a consequence, its class of so-called horizontal algebra morphisms with values in the PROS of endomorphisms of $B$ is a monoid, containing both the maps $E$ and $K$ standing for the distribution and the free cumulants of a random variable.
The main result of Section \ref{sec:shufflepointofview} is the following one, where $T_{\gtimes}(\mathcal{N}\mathcal{C})$ (resp. $T_{\gtimes}(\mathrm{Hom}(B)))$ is the space of all polynomials on non-crossing partitions (resp. on multilinear maps on $B$).
\begin{proposition*}[Proposition \ref{prop:shuffle}]
	$(\overline{\mathrm{Hom}}_{\mathrm{Coll}_{2}}(T_{\gtimes}(\mathcal{N}\mathcal{C}), T_{\gtimes}(\mathrm{Hom}(B))), \prec,\succ,\star)$ is a shuffle algebra.
\end{proposition*}

Thanks to the compatibility between the horizontal and vertical monoidal products, we can raise the notion of infinitesimal morphism in this context.
In Section \ref{sec:halfshuffle}, we compute explicitly the left and right half-shuffle exponentials. In particular, we show that both $K$ and $E$ are solutions of left half-shuffle fixed point equations. See Proposition \ref{prop:leftshuffle} as well as the Proposition \ref{prop:rightshuffle}:
\begin{equation}
  \label{eqn:fixedpointsun}
	K = \eta\circ \varepsilon + k \prec K,\quad
	E = \eta\circ\varepsilon + e \succ E.
\end{equation}
Each summand on the righthand side of equation \eqref{eqn:momentcumulants1} is interpreted as a value of solutions of half-shuffle fixed point equations.

Next, we define a structure for unshuffle Hopf algebra, similar to that of the operad of non-crossing partitions (adapted to the duoidal setting) on an operad of words insertions in Section \ref{sec:splitting}. We prove the following proposition.
\begin{proposition*}[Proposition \ref{prop:shuffleword}]
	$\left(\overline{\mathrm{Hom}}_{\mathrm{Coll}_{2}}(\mathcal{W}, T_{\gtimes}(\mathrm{Hom}(B)), \prec,\succ,\star\right)$ is a shuffle algebra.
\end{proposition*}
In addition, we define a map $Sp: \mathcal{W}\rightarrow T_{\gtimes}(\mathcal{N}\mathcal{C})$, the splitting map, induces a morphisms between the two shuffle algebras constructed previously, see proposition Proposition \ref{prop:unshuffleword} it reads:
\begin{equation}
	Sp(a_{1}\ldots a_{n}) = \sum_{\pi \in \mathrm{NC}(n)} \pi \otimes a_{1}\ldots a_{n}.
\end{equation}
Then, by pulling-back on the words insertions operad the first equation in \eqref{eqn:fixedpointsun}, we arrive at our main result, Proposition \ref{prop:momentcumulantsrelation} stating that $\eqref{eqn:momentcumulants1}$ is equivalent to the fixed point equation in $\mathcal{W}$:
\begin{equation}
  E = \eta \circ \varepsilon + k \prec E.
\end{equation}

\section{The gap-insertion operad of non-crossing partitions}
\label{sec:gapinsertion}

In this section we settle the algebraic structure on non-crossing partitions used throughout this work.
We start with a short reminder on collections and operads (both set and linear). Then, we formalize in this framework the idea of inserting a partition into the gaps of another partition. The reader is directed to \cite{ebrahimi2019operads} for a detailed exposition on this so-called gap-insertion operad and related structures. For general background on algebraic operads, both planar and symmetric and related concepts, we refer the reader to the monograph \cite{loday2012algebraic}.

\subsection{Set partitions}
\label{sec:partitions}

Let $X$ be a finite, linearly ordered set. A \emph{partition} of $X$ into disjoint sets (called blocks), $\pi_{i},~ 1 \leq i \leq k$, is denoted $\pi = \{\pi_{1},\ldots,\pi_{k}\}$.

An \emph{isomorphism between two set partitions} is a monotone bijection of the underlying linearly ordered sets compatible with the block structures. Then, any partition is equivalent to a partition of the linearly ordered set $\llbracket 1,n \rrbracket$ for some $n\in \mathbb{N}$. We call a set partition of $\llbracket 1,n \rrbracket$ a standard partition. It is convenient to work with the standard representative of each class.

For $k,n \in \mathbb{N}^{\star}$, we denote by $\textrm{SP}(k,n)$ the set of iso-classes of partitions of sets of $n$ elements into $k$ blocks. The set $\textrm{SP}(0,0)$ contains only the empty partition. We put
\begin{equation*}
	\textrm{SP} = \bigsqcup_{1 \leq k \leq n} \textrm{SP}(k,n),~
	\textrm{SP}(n) = \bigsqcup_{k\leq n} \textrm{SP}(k,n),~
	\textrm{SP}_{0} = \textrm{SP}(0,0) \sqcup \textrm{SP}.
\end{equation*}
Given a monotone inclusion of linearly ordered sets $X \subset Y$ and given a partition $\pi$ of $Y$, we write $\pi_{|X}$ for the trace of the partition of $\pi$ on $X$.

\begin{definition}
	Let $X$ be a non-empty finite subset of $\mathbb{N}$ (or a linearly ordered set $Y$). The \emph{convex hull} of $X$ is by definition Conv($X$)=$\llbracket \min(X),\max(X)\rrbracket$. We shall say that $X$ is convex if Conv($X$) $= X$. Any finite subset $X \subset \mathbb{N}$ decomposes uniquely as
	\begin{equation*}
		X=X_{1} \sqcup \cdots \sqcup X_{k}
	\end{equation*}
	with each $X_{i}$ convex and each $X_{i}\sqcup X_{j}$ \emph{not} convex for $i\neq j$. The $X_{i}$ are called the convex components of $X$.
\end{definition}

\begin{definition}[Non-crossing partitions]
	A partition $\pi = \{\pi_{1},\ldots,\pi_{k}\}$ is non-crossing if there are no $a,b \in \pi_{i}$ and no $c,d \in \pi_{j}$ with $i\neq j$ such that $a < c < b < d$.
\end{definition}
See Figure \ref{fig:noncrossingpartition} for examples of partitions. For a detailed overview of the algebraic structures of the set of non-crossing partitions, as well as an historical account, see \cite{simion2000noncrossing}. The notion of non-crossing partitions has first been introduced by Kreweras in the seminal article \cite{kreweras1972partitions}.

\begin{figure}
	\includegraphics[scale=1]{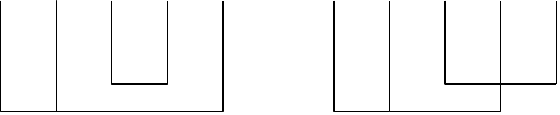}
	\caption{\label{fig:noncrossingpartition} Example of a non-crossing partition on the left, and a partition with a crossing on the right.}
\end{figure}

\begin{definition}[Interval partitions]
	\label{def:intervalpartitions}
	We say that a non-crossing partition $\pi\in\textrm{NC}$ is an \emph{interval partition} if all the blocks of $\pi$ are convex sets.
\end{definition}

\subsection{Algebraic planar operads}
\label{ssec:operads}

A collection $P$ is a sequence of vector spaces $(P(n))_{n\geq 1}$. A morphism between two collections is a sequence of linear morphisms $(\phi(n))_{n\geq1}$ with $\phi_{n}: P(n) \rightarrow P(n)$,~$n \geq 1$. The category of all collections is denoted $\mathrm{Coll}$. The tensor product $\sbt$ on the category $\mathrm{Coll}$ is the $2$-functor from $\mathrm{Coll} \times \mathrm{Coll}$ to $\mathrm{Coll}$ defined by:
\begin{align*}
	 & (P \sbt Q)(n) = \bigoplus_{\displaystyle{\substack{k \geq 1 \\n_{1}+\cdots+n_{k}=n } }}
	P(k) \otimes Q(n_{1}) \otimes \cdots \otimes Q(n_{k}),~
	 (f \sbt g)(n)= \bigoplus_{\displaystyle{\substack{k \geq 1   \\n_{1}+\cdots+n_{k}=n } }}
	f(k) \otimes g(n_{1}) \otimes \cdots \otimes g(n_{k}).
\end{align*}
The unit element for the tensor product $\sbt$ is the collection denoted by $\mathbb{C}_{\sbt}$ such that $\mathbb{C}_{\sbt}(n)=\delta_{n=1}\mathbb{C}$. An operad $\mathcal{P}$ is a monoid in the monoidal category $(\mathrm{Coll},\sbt, \mathbb{C})$, i.e., a triple $(P, \rho,\eta_{P})$ with
$$
	P \in \textrm{Coll},~\rho: P \sbt P  \rightarrow  P,~\eta_{P}: \mathbb{C} \rightarrow {P},
$$
satisfying $(\rho \sbt \textrm{id}_{P}) \circ \rho = (\textrm{id}_{P} \sbt \rho) \circ \rho$ and $(\eta_{P} \sbt \textrm{id}_{P}) \circ \rho = (\textrm{id}_{P} \sbt \eta_{P}) \circ \rho = \textrm{id}_{P}$. We use the notation $\sbt$ for the tensor product on collection to not confuse it with composition of functions. It is common to use the notation $\circ$ for an operadic composition:
\begin{equation}
	\rho(p \otimes (q_{1} \otimes \cdots \otimes q_{|p|})) = p \circ (q_{1}\otimes \cdots \otimes q_{|p|} )
\end{equation}
Accordingly, the notations $\circ_{i}$ for partial compositions:
\begin{equation}
	p\circ_{i}q = p \circ (1^{\otimes k-1} \otimes q \otimes \ldots 1^{|p|-k}),~ 1 \leq i \leq |p|.
\end{equation}
We should use these notations if there are no risks of confusion.
\subsection{Operad of partitions}
\label{ssec:partitionoperad}

A partition $\pi \in \textrm{SP}(n)$ is viewed as an operator with $n+1$ inputs. These inputs are the gaps between the elements of the partitioned set, including the front gap before $1$ and the back gap after $n$. We can insert $n+1$ partitions inside these gaps. It is clear that if $\pi$ is a non-crossing partition and we insert non-crossing partitions into the gaps of $\pi$ then the resulting partition is again non-crossing.

\begin{definition}
	\label{def:setpartitions}
	We set $\mathcal{S}\mathcal{P}(n) := \textrm{SP}(n-1)$. In particular, we have $\mathcal{S}\mathcal{P}(0)=\emptyset$ and $\mathcal{S}\mathcal{P}(1)=\{\emptyset\}$. The empty partition is the operad unit. Let $\pi$ be a partition and $(\alpha_{1},\ldots,\alpha_{|\pi|})$ a sequence of set partitions. The composition $\rho_{\mathcal{S}\mathcal{P}}(\pi \otimes  \alpha_{1} \otimes \cdots \otimes \alpha_{|\pi|})$ is obtained by inserting each partition $\alpha_{i}$ in between the two integers $i$ and $i+1$, $~i \leq 1$. In symbols:
	\begin{equation*}
		\rho_{\mathcal{S}\mathcal{P}}(\pi \otimes \alpha_{1} \otimes \cdots \otimes \alpha_{|\pi|})
		= \bigcup_{i=1}^{|\pi|}\{i-1+b,~b \in \pi_{i}\} \cup \tilde{\pi}
	\end{equation*}
	where $\tilde{\pi}$ is the partition of $\{|\pi_{1}|, |\pi_{1}|+|\pi_{2}|,\ldots,|\pi_{1}|+\cdots+|\pi_{n}|\}$ induced by $\pi$.
\end{definition}

\begin{figure}[!h]
	\centering
	\includegraphics[scale=0.85]{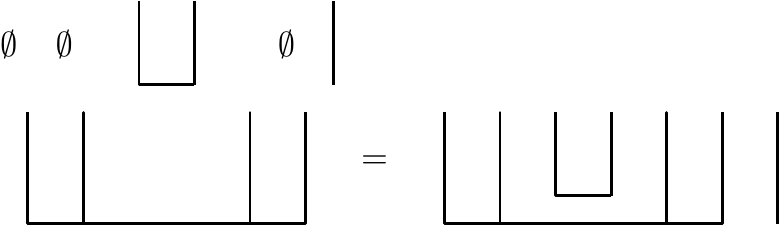}
	\caption{\label{fig:operadinsertion} Example of a composition in the gap-insertion operad $\mathcal{N}\mathcal{C}$.}
\end{figure}

\begin{lemma}
	The sequence $\mathcal{N}\mathcal{C} = (\mathcal{N}\mathcal{C}(n))_{n\geq 1}$ with $\mathcal{N}\mathcal{C}(n-1) = \mathrm{NC}(n)$ defines a set operad called the non-crossing gap-insertion operad when equipped with the composition law $\rho_{\mathcal{N}\mathcal{C}} = \rho_{\mathcal{S}\mathcal{P}}{}_{|\mathcal{N}\mathcal{C} \sbt \mathcal{N}\mathcal{C}}$.
\end{lemma}

The two set operadic structures $\rho_{\mathcal{S}\mathcal{P}}$ and $\rho_{\mathcal{N}\mathcal{C}}$ induce linear operadic structures on the free vector spaces spanned by SP, respectively NC. In the following, we shall not distinguish between them.

For each integer $n\geq 1$, let $\textrm{Int}(n)$ be the set of all interval partitions of $\llbracket 1,n-1 \rrbracket$. Set $\textrm{Int} = \bigcup_{n\geq 0} \textrm{Int}(n)$, then $\textrm{Int}$ is a sub-collection of $\mathcal{N}\mathcal{C}$.

The gap-insertion operad of non-crossing partitions admits the following presentation in terms of generators and relations.

\begin{lemma}[Proposition 3.1.4 in \cite{ebrahimi2019operads}]
	For any $n\geq 1$, we put $1_{n+1} = \{\llbracket 1,n\rrbracket\}$. Then the operad $(\mathcal{N}\mathcal{C}, \rho_{\mathcal{N}\mathcal{C}})$ is generated by the elements $1_{n}$, $n \geq 1$ with the relation:
	\begin{equation*}
		\forall m,n \geq 1,\quad 1_{m} \circ_{m} 1_{n} = 1_{n} \circ_{1} 1_{m}.
	\end{equation*}
\end{lemma}

The gap-insertion operad implements algebraically the \emph{nesting} of blocks of a partition. Let $a \in \mathcal{A}$ a random variable, we defined in the introduction for each $n\geq 1$ the map from $B^{\otimes n}$ to $B$:
\begin{equation*}
	E_{n+1}(b_{0},\ldots,b_{n}) = E(a^{\otimes n})(b_{0},\ldots,b_{n}) = E(b_{0}ab_{1}a\cdots ab_{n})
\end{equation*}
with $b_{0},\ldots,b_{n} \in B$. Now since $E$ is $B$-$B$ bimodule map, we get
\begin{align*}
	\label{eqn:relation1}
	 & E_{1}(b_{0}) = b_{0}                             \\
	 & E_{n} \circ_{n} E_{m} = E_{m} \circ_{1} E_{n}.
\end{align*}
In fact, we have, for $b_{1},\ldots,b_{n+m} \in B$
\begin{align*}
	 (E_{n} \circ_{n} E_{m})(b_{0},\ldots,b_{n+m}) &= E(b_{0}ab_{1}\cdots a E_{m}(b_{n},\ldots,b_{n+m})= E(b_{0}ab_{1}\cdots a)E(b_{n}a\cdots b_{n+m}) \\
	  & = E(E(b_{0}ab_{1}\cdots a)b_{n}a\cdots b_{n+m})= E(E(b_{0}ab_{1}\cdots ab_{n})a\cdots b_{n+m})\\
		&=(E_{m}\circ_{1}E_{n})(b_{0},\ldots,b_{n+m}).
\end{align*}
Hence, there exists an unique operadic morphism $E : \mathcal{N}\mathcal{C} \rightarrow \mathrm{Hom}(B)$ such that $E(1_{n}) = E_{n}$. The free cumulants of $E$ enjoy the same property: there exists an unique operadic morphism $K:\mathcal{N}\mathcal{C} \rightarrow \mathrm{Hom}(B)$ such that
\begin{equation*}
	K(1_{n})(b_{0},\ldots,b_{n}) = k_{n}(b_{0}ab_{1},\ldots,a b_{n}).
\end{equation*}

\smallskip

Non-crossing partitions are central to free probability theory while in boolean probability theory, interval partitions are the main combinatorial objects. The collection of interval partitions is not a sub-operad of $\mathcal{N}\mathcal{C}$.

To implement the shuffle point of view for operator-valued probability theory, we will adapt the construction given in \cite{ebrahimi2019operads} to the operator-valued case. In \cite{ebrahimi2019operads}, the authors start with the definition of a bialgebraic structure on the vector space $N$ of non-commutative polynomials in non-crossing partitions. On $N$, the operadic composition $\rho_{\mathcal{N}\mathcal{C}}$ induces a coproduct $\Delta: N \rightarrow N \otimes N$:
\begin{equation}
	\Delta(\pi) = \sum_{\alpha \circ \beta_{1},\ldots,\beta_{|\alpha|}} \alpha \otimes \beta_{1}\cdots\beta_{|\alpha|}.
\end{equation}

In the scalar case, the moment and the free cumulants of a random variable, seens as functions on the set of non-crossing partitions NC are implemented as characters on $N$ which are elements of the convolution group associated with the coproduct $\Delta$.

In comparison, for the operator-valued case, we construct two algebra morphisms $T_{\gtimes}(E)$ and $T_{\gtimes}(K)$ from $N$ to the space of non-commutative polynomials on multilinear maps on $B$ and extending to $N$ the maps $E$ and $K$ constructed previously.

In addition, since $K$ and $E$ are operadic morphisms, we see that two natural compositions of words on non-crossing partitions should be considered: an concatenation (which will be called horizontal composition) and a composition extending the operadic structure on non-crossing partitions (which will be called vertical). The following section evolves on this idea using the notion of duoidal category.

Notice that we have considered so far the case of a single random variable, but the construction of the operadic morphisms $K$ and $E$ extends readily to the multivariate case by considering \emph{coloured partitions}.

\section{The duoidal category of bicollections}
\label{sec:duoidal}

Elements of a collection are operators with many inputs and a single output. The operadic structure models compositions between these operators. In many branches of mathematics, ranging from probability theory, both classical and non-commutative, to gauge theory and quantum groups algebraic structures with products and co-products that stand for merging, respectively cutting, processes have become popular. The framework of operads is however too narrow to treat such structures completely.

Indeed, it turns out to be important to be able to handle operations with multiple in- and outputs. After the work of Adams and McLane,  \cite{adams1978infinite, mac1971categorical} and Vallette \cite{vallette2007koszul}, the right algebraic framework appears to be the one of properads, props and their extensions.

The construction we expose in the section is reminiscent of the props setting, but is in fact much simpler as it does not involve actions of the symmetric groups.

We introduce now the prominent algebraic structure to the present work, i.e., the category of bicollections endowed with two balanced monoidal structures. In the literature, such a category is called a \emph{duoidal}\footnote{\url{https://ncatlab.org/nlab/show/duoidal+category}} category or a $2$-monoidal category. The interested reader is directed to monograph \cite{aguiar2010monoidal} for a comprehensible introduction to $2$-monoidal categories. This section focuses on the so-called \emph{laxity property} stated in \eqref{eqn:laxity}. It is beyond the scope for the present work to provide the reader with a detailed account on the notion of duoidal category. Nevertheless, for the sake of completeness, we will give the definition of such a category, without fully commenting on it.

\begin{definition}[Duoidal category]
	\label{def:duoidalcategory}
	\emph{A duoidal category}, or $2$-monoidal category, is a category $\mathcal{C}$ endowed with a monoidal structure $(\gtimes, E_{\gtimes})$, together with an additional monoidal structure $(\oboxtimes,E_{\oboxtimes})$ such that $\oboxtimes: C \times C \rightarrow C$ and $E_{\oboxtimes} : 1 \rightarrow C$ are lax monoidal functors with respect to $(\gtimes,E_{\gtimes})$ and the coherence axioms of $(\oboxtimes, E_{\oboxtimes})$ are monoidal natural transformation with respect to $(\gtimes, E_{\gtimes})$. The laxity of $(\oboxtimes, E_{\oboxtimes})$ consists of natural transformations
	\begin{align*}
		 & \label{eqn:genlax}(C_{1} \oboxtimes C_{2}) \gtimes (C_{3}\oboxtimes C_{4}) \overset{R_{C_{1},C_{2},C_{3},C_{4}}}{\longrightarrow} (C_{1} \gtimes C_{3})\oboxtimes (C_{2}\gtimes C_{4}),
	\end{align*}
	together with morphisms $E_{\gtimes} \rightarrow E_{\gtimes} \oboxtimes E_{\gtimes},\quad E_{\oboxtimes} \gtimes E_{\oboxtimes} \rightarrow E_{\oboxtimes},\quad E_{\gtimes}\rightarrow E_{\oboxtimes}$.

\end{definition}

%

\smallskip

The main result of this section is the following one.

\begin{proposition*}
	The category of bicollections is a duoidal category.
\end{proposition*}


Historically, there are at least two other notions similar to that of a duoidal category, introduced in earlier work. The first one is the notion of \emph{two fold monoidal categories} of Baltenau and Fiedorowcz \cite{balteanu1996coherence,kock2007note}. In such a category, the two monoidal structures are required to be \emph{strict} (this property holds for the category of bicollections, see below) but also to share a common unit object (which is not the case for the category of bicollections). Later Forcey, Siehler and Sowers \cite{forcey2007operads} improved upon this notion by removing the strictness assumption, allowing the unit objects to be different, but requiring stronger assumptions on the units. This fails for the duoidal category of bicollections.

We are ultimately interested in the categories of algebras in a duoidal category with respect to one of the two monoidal products (or the two at a time). A PROS, as used for example in \cite{bultel2016combinatorial}, is both an algebra in the monoidal category ($\mathcal{C}, \oboxtimes, E_{\oboxtimes}$) and in ($\mathcal{C}, \gtimes, E_{\gtimes}$). In addition, we require for the two compositions to be compatible in a certain sense.

The very first example of a $2$-monoidal category is provided by a symmetric monoidal category, or more generally by a braided monoidal category. In that case, the two monoidal structures coincide. In this case, the category of PROS in a braided monoidal category contains braided commutative algebras.

Among (braided) commutative algebras, we find commutative bialgebras and commutative Hopf algebras. A commutative Hopf algebra $H$ provides a functor $F_{H}$ from the category of commutative algebras to the category of groups:
$$F_{H}(A) = \mathrm{Hom}_{\textrm{Alg}}(H,A)$$
with $A$ a commutative algebra, whereas a bialgebra provides, by the same formula, a functor from the category of commutative algebra to the category of monoids.
In that respect and in view of the application to non-commutative probability theory, it is thus natural to look for PROS in a duoidal category that are also coalgebras for the tensor product $\oboxtimes$ and furthermore those that can be endowed with an antipodal map.

\subsection{The horizontal and vertical tensor product}
\label{subsect:horizontalvertical}

In this section, we formalize the idea of composing operators with multiple in- and outputs (many-to-many operators). Branching outputs of an operator to the inputs of another one defines a product on a space of many-to-many operators. We refer to this product by the terminology \emph{vertical}. There is another way to compose such operators: concatenating the outputs (resp. the inputs) of two operators. This is the \emph{horizontal} product. We give a definition of a PROS in the category of bicollections using the language of $2$-monoidal categories (or duoidal categories).

\begin{definition}[Bicollection]
	\label{def:bigradedcollection}
	A bicollection is a two parameters family of vector spaces $$P = (P(n,m))_{n,m \geq 0}.$$ A morphism between two bicollections $P$ and $Q$ is a family of linear maps $\phi(n,m):P(n,m) \rightarrow Q(n,m)$. The category of all bicollections is denoted $\mathrm{Coll}_{2}$.
\end{definition}


\begin{definition}[Horizontal tensor product]
	\label{def:horibicollection}
	The horizontal tensor product $\gtimes$ is the functor $\gtimes : \mathrm{Coll}_{2}\times\mathrm{Coll}_{2} \rightarrow \mathrm{Coll}_{2}$ defined by:
	\begin{align*}
		 & (P\gtimes Q)(n,m) = \displaystyle\bigoplus_{\substack{n_{1}+n_{2}=n \\ m_{1}+m_{2}=m}} P(n_{1},m_{1}) \otimes  Q(n_{2},m_{2}),~
		(f \gtimes g)(n, m) = \displaystyle\bigoplus_{\substack{n_{1}+n_{2}=n             \\ m_{1}+m_{2}=m}} f(n_{1},m_{1}) \otimes g(n_{2},m_{2}).
	\end{align*}
	The identity element for the horizontal tensor product $\gtimes$ is the bicollection $\pmb{\mathbb{C}}_{\gtimes}(n,m) = \delta_{n,m=0}\mathbb{C}$.
\end{definition}

\begin{definition}[Vertical tensor product]
	\label{def:vertbicollection}
	The tensor product $\oboxtimes$ on the category Coll$_{2}$ is defined by:
	\begin{align*}
		 & (P\oboxtimes Q)(n,m) = \bigoplus_{k \geq 0} P(n,k) \otimes  Q(k,m),~(f \oboxtimes g)(n,m) = \bigoplus_{k \geq 0} f(n,k) \otimes g(k,m).
	\end{align*}
	The identity element for the tensor product $\oboxtimes$ is the bicollection $\pmb{\mathbb{C}}_{\oboxtimes}(n,m) = \delta_{n=m}\mathbb{C}$.
\end{definition}

Fundamental examples of bicollections are obtained by taking polynomials on operators in a given collection. Pick $P = (P_{n})_{n\geq 1}$ a collection, and define a bicollection $P$ by
\begin{equation}
	\label{eqn:wordscollection}
	P(m,n) = \bigoplus_{\substack{k_{1} + \cdots + k_{n}=m}} P_{k_{1}}\gtimes\cdots\gtimes P_{k_{n}} \textrm{ and set } 	T_{\gtimes}(P) = \mathbb{C}1 \oplus P.
\end{equation}
with $1$ being an element with $0$ inputs and zero $0$ outputs. All bicollections we work with are of the form \eqref{eqn:wordscollection}. For example, considering $P = (\mathrm{Hom}(B^{\gtimes n},B))_{n\geq 0}$ the bicollection $T_{\gtimes}(\mathrm{Hom}(B))$ plays a prominent role in the sequel.
We have already encountered another example of bicollection whose homogeneous components are spanned by forests with a certain number of trees and leaves.

In Figure~\ref{fig:hvel} the reader will find a pictorial description of elements in the horizontal and vertical tensor products. In the vertical tensor product, the number of inputs of the operator on the lower level matches the number of outputs of the operator on the upper level. In comparison with the vertical tensor product introduced in \cite{vallette2007koszul}, the tensor product $P\oboxtimes Q$ we introduce here is a sum over planar 2-level diagrams with only one vertex on each level (see Fig. \ref{fig:hvel}). In \cite{vallette2007koszul}, the author considers bisymmetric sequences of vector spaces, and the monoidal structure involves either a sum over $2-$level connected graphs for properads or on connected graphs for props.

It is easy to design a generalization of the vertical tensor product: we sum over connected planar diagrams connecting vertices placed on the integer points of the lines $\mathbb{R} \times \{0\}$ to vertices placed on the line $\mathbb{R} \times \{1\}$.

Let us mention that the vertical tensor $\oboxtimes$ has also been considered by Bultel and Giraudo in \cite{bultel2016combinatorial}, in which the authors define Hopf algebraic type structures on PROS.
The vertical tensor product for a pair of bicollections of the form \eqref{eqn:wordscollection}, can also be depicted as a sum over (not-necessarily connected) two level planar graphs, obtained as concatenation of corollas.

\begin{figure}[!ht]
	\label{fig:hvel}
\begingroup%
  \makeatletter%
  \providecommand\color[2][]{%
    \errmessage{(Inkscape) Color is used for the text in Inkscape, but the package 'color.sty' is not loaded}%
    \renewcommand\color[2][]{}%
  }%
  \providecommand\transparent[1]{%
    \errmessage{(Inkscape) Transparency is used (non-zero) for the text in Inkscape, but the package 'transparent.sty' is not loaded}%
    \renewcommand\transparent[1]{}%
  }%
  \providecommand\rotatebox[2]{#2}%
  \newcommand*\fsize{\dimexpr\f@size pt\relax}%
  \newcommand*\lineheight[1]{\fontsize{\fsize}{#1\fsize}\selectfont}%
  \ifx\svgwidth\undefined%
    \setlength{\unitlength}{242.90375172bp}%
    \ifx\svgscale\undefined%
      \relax%
    \else%
      \setlength{\unitlength}{\unitlength * \real{\svgscale}}%
    \fi%
  \else%
    \setlength{\unitlength}{\svgwidth}%
  \fi%
  \global\let\svgwidth\undefined%
  \global\let\svgscale\undefined%
  \makeatother%
  \begin{picture}(1,0.42924327)%
    \lineheight{1}%
    \setlength\tabcolsep{0pt}%
    \put(0,0){\includegraphics[width=\unitlength,page=1]{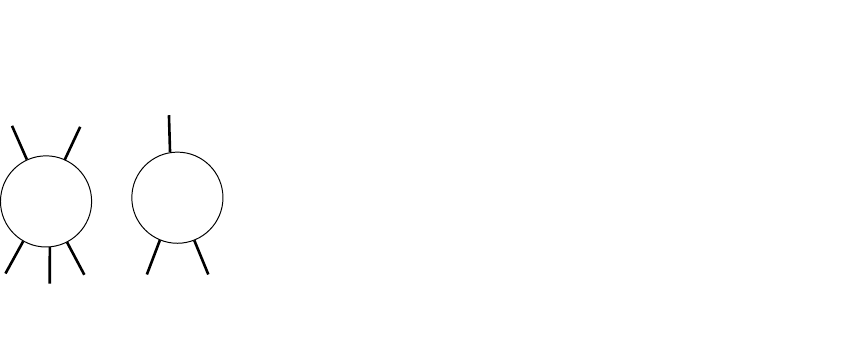}}%
    \put(0.10649765,0.17787041){\color[rgb]{0,0,0}\makebox(0,0)[lt]{\lineheight{1.25}\smash{\begin{tabular}[t]{l}$\gtimes$\end{tabular}}}}%
    \put(0,0){\includegraphics[width=\unitlength,page=2]{hvel.pdf}}%
    \put(0.3969756,0.20344081){\color[rgb]{0,0,0}\makebox(0,0)[lt]{\lineheight{1.25}\smash{\begin{tabular}[t]{l}$\oboxtimes$\end{tabular}}}}%
    \put(0,0){\includegraphics[width=\unitlength,page=3]{hvel.pdf}}%
  \end{picture}%
\endgroup%

	\caption{\label{fig:hvel} On the left, we have elements in the horizontal $\gtimes$ and vertical $\oboxtimes$ tensor product s. On the right, we have a bundle.}
\end{figure}

\begin{remarque}
	The tensor product $\gtimes$ is a symmetric one, whereas $\oboxtimes$ is not. Neither the horizontal nor the vertical tensor product come with injections and the units for these two tensor products are not initial objects.
\end{remarque}


In the sequel, to distinguish elements in the tensor products $A\oboxtimes B$ or $A\gtimes B$, we use the notation $a \oboxtimes b$, respectively $a\gtimes b$. In the first case, the notation emphasizes that fact that the number of inputs of $a$ matches the number of outputs of $b$. The standard monoidal tensor product on the category Vect$_{\mathbb{C}}$ of vector spaces is denotes $\otimes$.
\begin{proposition}
	\label{lem:lax}
	Let $C_{i},~ 1 \leq i \leq 4$ be four  bicollections, then
	\begin{equation}
		\label{eqn:laxity}
		\left(C_{1} \oboxtimes C_{2}\right) \gtimes \left(C_{3}\oboxtimes C_{4} \right) \hookrightarrow \left(C_{1} \gtimes C_{3}\right) \oboxtimes \left(C_{2}\gtimes C_{4}\right).
	\end{equation}
	The morphism is denoted by $R_{C_{1},C_{2},C_{3},C_{4}}$. With $C$ a collection, one has:
	\begin{equation}
		\left(C_{1} \oboxtimes T_{\gtimes}(C)\right) \gtimes \left(C_{2}\oboxtimes T_{\gtimes}(C) \right) \simeq \left(C_{1} \gtimes C_{2}\right) \oboxtimes T_{\gtimes}(C).
	\end{equation}
\end{proposition}
\begin{proof} Let $C_{1},C_{2},C_{3}$ and $C_{4}$ be four bicollections. Let $p^{1},p^{2},p^{3},p^{4}$ be elements of respectively, $C_{1},C_{2},C_{3}$ and $C_{4}$ with the number of outputs of $p^{2}$ matching the number of inputs of $p^{1}$ and the same for $p^{3}$ and $p^{4}$. We denote by $S$ the braiding of the symmetric monoidal category
$(\otimes, \textrm{Vect}_{\mathbb{C}})$ Next, we define $$R_{C_{1},C_{2},C_{3},C_{4}}: C_{1}\oboxtimes C_{2} \gtimes C_{3} \oboxtimes C_{4} \rightarrow C_{1}\oboxtimes C_{3} \gtimes C_{2} \oboxtimes C_{4}$$ by
\begin{align*}
		 & R_{C_{1},C_{2},C_{3},C_{4}} \left((p^{1} \oboxtimes p^{2}) \gtimes (p^{3}\oboxtimes p^{4})\right) = R_{C_{1},C_{2},C_{3},C_{4}} \left((p^{1} \otimes p^{2}) \otimes (p^{3}\otimes p^{4})\right)      \\
		 & \hspace{3cm} =(\textrm{id} \otimes S \otimes \textrm{id})(p^{1}\otimes p^{2} \otimes p^{3} \otimes p^{4}) = p^{1}\otimes p^{3} \otimes p^{2} \otimes p^{4} = (p^{1}\gtimes p^{3})\oboxtimes (p^{2}\gtimes p^{4}).
\end{align*}
	First, it is easy to see that $R_{C_{1},C_{2},C_{3},C_{4}}$, is well defined, and  if extended linearly it becomes a morphism of bicollections. Moreover it is injective. However, it is not surjective. In particular, the image of $R$ is the span of the elements $(p^{1}\gtimes p^{3}) \oboxtimes (p^{2}\gtimes p^{4})$ with a perfect match between the inputs of $p^{3}$ and the outputs of $p^{4}$ on one hand, the inputs of $p^{1}$ and the outputs of $p^{2}$ on the other hand.

	To prove the second assertion, we first notice that $T_{\gtimes}(C)$ is endowed with an unital algebraic structure, given by the concatenation of words, for which $1 \in T_{\gtimes}(C)$ is the unit. We denote by $m: T_{\gtimes}(C)\gtimes T_{\gtimes}(C) \rightarrow T_{\gtimes}(C)$ the algebra map.
	We denote by $q_{1} \cdots q_{s}$ the product of operators $q_{1},\ldots,q_{s}$ in $T_{\gtimes}(C)$. For brevity, we also use the notation $|p|$ for the number of inputs of an operator $p$ in a  bicollection. Define the map
	$$
	\tilde{R}_{C_{1},T_{\gtimes}(C),C_{2},T_{\gtimes}(C)}: (C_{1}\gtimes C_{2}) \oboxtimes T_{\gtimes}(C) \rightarrow \left(C_{1}\oboxtimes T_{\gtimes}(C)\right) \gtimes (C_{2} \oboxtimes T_{\gtimes}(C))
	$$
	by:
	\begin{equation*}
		\tilde{R}_{C_{1},T_{\gtimes}(C),C_{2},T_{\gtimes}(C)}((p^{1} \gtimes p^{2})\oboxtimes (p_{1}\gtimes \cdots \gtimes p_{|p^{1}|+|p^{2}|}))
		= (p^{1} \oboxtimes (p_{1} \cdots p_{|p^{1}|})) \gtimes (p^{2}\oboxtimes(p_{|p^{1}|+1} \cdots p_{|p^{1}|+|p^{2}|})),
	\end{equation*}
	with the convention that if $|p^{1}|=0$ or $|p^{2}|=0$, then we set $p_{1}\cdots p_{|p_{1}|} = 1$, and, respectively, $p_{|p_{1}|+1}\cdots p_{|p_{1}|+|p_{2}|}=1$.
	We should prove first that
	\begin{equation}
		\label{eqn:iso1}
		\tilde{R}_{C_{1},T_{\gtimes}(C),C_{2},T_{\gtimes}(C)} \circ ((\textrm{id}_{C_{1}} \gtimes \textrm{id}_{C_{2}})\oboxtimes m) \circ R_{C_{1},T_{\gtimes}(C),C_{2},T_{\gtimes(C)}} = \textrm{id}.
	\end{equation}
	Notice that $1 \in T_{\gtimes}(C)$ is the unique element in $T_{\gtimes}(C)$ with zero outputs (also the unique one with zero inputs). Assume first that $|p^{1}|, |p^{2}| > 0$. The left hand side of \eqref{eqn:iso1} applied to $$p^{1}\oboxtimes (p_{1}\gtimes \cdots \gtimes p_{|p^{1}|}) \gtimes (p^{2}\oboxtimes (q_{1}\gtimes\cdots\gtimes q_{|p^{2}|}))$$ gives:
	\begin{align*}
		 & \tilde{R}_{C_{1},T_{\gtimes}(C),C_{2},T_{\gtimes}(C)}\left( (p^{1}\gtimes p^{2}) \oboxtimes (p_{1} \gtimes \cdots \gtimes p_{|p^{1}|} \gtimes q_{1}\gtimes\cdots\gtimes q_{|p^{2}|})\right) \\
		 & \hspace{5cm} = p^{1}\oboxtimes (p_{1}\gtimes \cdots \gtimes p_{|p^{1}|}) \gtimes (p^{2}\oboxtimes (q_{1}\gtimes\cdots\gtimes q_{|p^{2}|})).
	\end{align*}
	Now assume that $|p^{1}|=0$. Then, the left hand side of \eqref{eqn:iso1} applied to $(p^{1}\oboxtimes 1) \gtimes (p^{2}\oboxtimes (q_{1}\gtimes\cdots\gtimes q_{|p^{2}|}))$ gives:
	\begin{align*}
		\tilde{R}_{C_{1},T_{\gtimes}(C),C_{2},T_{\gtimes}(C)}\left( (p^{1}\gtimes p^{2}) \oboxtimes (q_{1}\cdots q_{|p^{2}|})\right) = (p^{1} \oboxtimes 1) \gtimes (p^{2} \oboxtimes (q_{1}\cdots q_{|p^{2}|})).
	\end{align*}
	Finally, the same line of thoughts applies to prove that
	\begin{equation}
		\label{eqn:iso2}
		((\textrm{id}_{C_{1}} \gtimes \textrm{id}_{C_{2}})\oboxtimes m) \circ R_{C_{1},T_{\gtimes}(C),C_{2},T_{\gtimes(C)}} \circ \tilde{R}_{C_{1},T_{\gtimes}(C),C_{2},T_{\gtimes}(C)} = \textrm{id}.
	\end{equation}

\end{proof}
The natural transformation $R$ is sometimes called \emph{exchange law} and the relation \eqref{eqn:laxity} is called \emph{middle-four interchange}.
\begin{figure}[!h]
\begingroup%
  \makeatletter%
  \providecommand\color[2][]{%
    \errmessage{(Inkscape) Color is used for the text in Inkscape, but the package 'color.sty' is not loaded}%
    \renewcommand\color[2][]{}%
  }%
  \providecommand\transparent[1]{%
    \errmessage{(Inkscape) Transparency is used (non-zero) for the text in Inkscape, but the package 'transparent.sty' is not loaded}%
    \renewcommand\transparent[1]{}%
  }%
  \providecommand\rotatebox[2]{#2}%
  \newcommand*\fsize{\dimexpr\f@size pt\relax}%
  \newcommand*\lineheight[1]{\fontsize{\fsize}{#1\fsize}\selectfont}%
  \ifx\svgwidth\undefined%
    \setlength{\unitlength}{197.42634012bp}%
    \ifx\svgscale\undefined%
      \relax%
    \else%
      \setlength{\unitlength}{\unitlength * \real{\svgscale}}%
    \fi%
  \else%
    \setlength{\unitlength}{\svgwidth}%
  \fi%
  \global\let\svgwidth\undefined%
  \global\let\svgscale\undefined%
  \makeatother%
  \begin{picture}(1,0.53743362)%
    \lineheight{1}%
    \setlength\tabcolsep{0pt}%
    \put(0,0){\includegraphics[width=\unitlength,page=1]{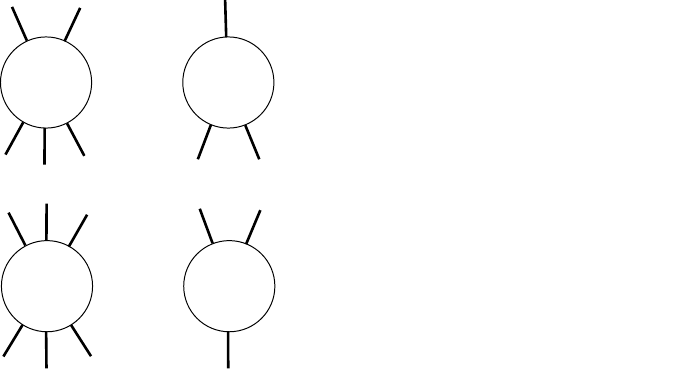}}%
    \put(0.03992777,0.2540053){\color[rgb]{0,0,0}\makebox(0,0)[lt]{\lineheight{1.25}\smash{\begin{tabular}[t]{l}$\oboxtimes$\end{tabular}}}}%
    \put(0.30489998,0.2540053){\color[rgb]{0,0,0}\makebox(0,0)[lt]{\lineheight{1.25}\smash{\begin{tabular}[t]{l}$\oboxtimes$\end{tabular}}}}%
    \put(0.18686314,0.25644744){\color[rgb]{0,0,0}\makebox(0,0)[lt]{\lineheight{1.25}\smash{\begin{tabular}[t]{l}$\gtimes$\end{tabular}}}}%
    \put(0.78084605,0.40053377){\color[rgb]{0,0,0}\makebox(0,0)[lt]{\lineheight{1.25}\smash{\begin{tabular}[t]{l}$\gtimes$\end{tabular}}}}%
    \put(0.77519649,0.10422072){\color[rgb]{0,0,0}\makebox(0,0)[lt]{\lineheight{1.25}\smash{\begin{tabular}[t]{l}$\gtimes$\end{tabular}}}}%
    \put(0.77094185,0.25617605){\color[rgb]{0,0,0}\makebox(0,0)[lt]{\lineheight{1.25}\smash{\begin{tabular}[t]{l}$\oboxtimes$\end{tabular}}}}%
    \put(0,0){\includegraphics[width=\unitlength,page=2]{rmap.pdf}}%
    \put(0.49592966,0.28141148){\color[rgb]{0,0,0}\makebox(0,0)[lt]{\lineheight{1.25}\smash{\begin{tabular}[t]{l}R\end{tabular}}}}%
  \end{picture}%
\endgroup%

	\caption{\label{fig:rmap} Drawing of the action of the natural functor $R$. On the left hand side, the vertical products are taken first between vertically arranged pairs, then we take the horizontal product. On the right hand side, we do the opposite.}
\end{figure}

A remark on the graphical presentation of the exchange law in Figure \ref{fig:rmap}. In \cite{aguiar2010monoidal}, the authors rather than using the symbols $\gtimes$ and $\oboxtimes$ replace them a by a simple straight line to indicate the operation that precede. Other authors follow a different convention and choose to represent by a straight line the last operation. In that case, on the left hand side in Figure \ref{fig:rmap}, the horizontal line of symbol is replaced by a vertical line following that convention, and correspondingly for the right hand side.

The family of morphisms $\{R_{C_{1},C_{2},C_{3},C_{4}}, C_{i} \in \mathrm{Coll}_{2}\}$ define a natural transformation between the two functors $\gtimes \circ \oboxtimes \times \oboxtimes$ and $\oboxtimes \circ \gtimes \times \gtimes$. In fact, pick four morphisms $f_{i}:C_{i}\rightarrow D_{i},~ 1 \leq i \leq 4$, the diagram  in Figure \ref{fig:natr} is a commutative diagram.

\begin{figure}[!h]
	\centering
	\begin{tikzcd}
		(C_{1} \oboxtimes C_{2}) \gtimes (C_{3} \oboxtimes C_{4}) \arrow{r}[yshift=2ex]{(f_{1}\oboxtimes f_{2})\gtimes (f_{3}\oboxtimes f_{4})}\arrow{d}{R_{C_{1},C_{2},C_{3},C_{4}}} & (D_{1} \oboxtimes D_{2}) \gtimes (D_{3} \oboxtimes D_{4}) \arrow{d}{R_{D_{1},D_{2},D_{3},D_{4}}}\\
		(C_{1} \gtimes C_{3}) \oboxtimes (C_{2} \gtimes C_{4}) \arrow{r}[yshift=-2ex,swap]{(f_{1}\gtimes f_{3}) \oboxtimes (f_{2}\gtimes f_{4})} & (D_{1}\gtimes D_{3}) \oboxtimes (D_{2}\gtimes D_{4})
	\end{tikzcd}
	\caption{\label{fig:natr} Naturality of $R$.}
\end{figure}

\subsection{Monoids and comonoids}

We present here general results on categories of algebras in a duoidal category. All proofs can be found in the monograph \cite[Chapt.~6, Sects.~6.5, 6.6]{aguiar2010monoidal}. We define the notion of \raisebox{-1pt}{$\oboxtimes$}$\gtimes$-bialgebra and the notion of \raisebox{-1pt}{$\oboxtimes$}$\gtimes$-Hopf algebras. We use the terminology in reference to the one used in the context of braided symmetric monoidal categories (as explained in the introduction).


\medskip

We denote by $\mathrm{Alg}_{\oboxtimes}$ the category of unital complex associative algebras in the monoidal category $(\mathrm{Coll}_{2},\oboxtimes, \pmb{\mathbb{C}}_{\oboxtimes})$ and $\mathrm{Alg}_{\gtimes}$ the category of complex associative algebras in the monoidal category $(\textrm{Coll}_{2},\gtimes,\pmb{\mathbb{C}}_{\gtimes})$. We write an horizontal, respectively a vertical algebra, as a triplet $(A,m_{\gtimes}^{A}, \eta^{A}_{\gtimes})$, respectively $(A,m_{\oboxtimes}^{A},\eta^{A}_{\oboxtimes})$, with:
\begin{equation*}
	m_{\gtimes}^{A} : A \gtimes A \rightarrow A,~\eta_{\gtimes}^{A} : \pmb{\mathbb{C}_{\gtimes}} \rightarrow A, \quad m_{\oboxtimes}^{A}: A \oboxtimes A \rightarrow A,~ \eta_{\oboxtimes}^{A}:\pmb{\mathbb{C}}_{\oboxtimes} \rightarrow A.
\end{equation*}
Notice that the unit $\pmb{\mathbb{C}}_{\oboxtimes}$ of the vertical tensor product $\oboxtimes$ is an algebra in the monoidal category $(\mathrm{Coll}_{2}, \gtimes, \pmb{\mathbb{C}}_{\gtimes})$:
\begin{equation*}
	m_{\gtimes}^{\pmb{\mathbb{C}}_{\oboxtimes}}: \pmb{\mathbb{C}}_{\oboxtimes} \gtimes \pmb{\mathbb{C}}_{\oboxtimes}\rightarrow \pmb{\mathbb{C}}_{\oboxtimes},~ 1_{n}\gtimes1_{m} \mapsto 1_{n+m},\quad \eta_{\gtimes}^{\pmb{\mathbb{C}}_{\oboxtimes}} : \pmb{\mathbb{C}_{\gtimes}} \rightarrow \pmb{\mathbb{C}}_{\oboxtimes},~ 1_{0} \rightarrow 1_{0}.
\end{equation*}
Likewise for the unit $\pmb{\mathbb{C}}_{\gtimes}$, it is a coalgebra in the monoidal category (Coll$_{2}$, $\oboxtimes, \mathbb{C}_{\oboxtimes}$):
\begin{equation*}
	\Delta^{\oboxtimes}_{\pmb{\mathbb{C}}_{\gtimes}} : \pmb{\mathbb{C}}_{\gtimes} \rightarrow \pmb{\mathbb{C}}_{\gtimes} \oboxtimes \pmb{\mathbb{C}}_{\gtimes},~ 1_{0} \rightarrow 1_{0} \oboxtimes 1_{0},~\varepsilon_{\pmb{\mathbb{C}_{\gtimes}}}^{\oboxtimes} = \eta_{\gtimes}^{\pmb{\mathbb{C}}_{\oboxtimes}}
\end{equation*}
\begin{proposition}[Proposition 6.3.5 in \cite{aguiar2010monoidal}]
	\label{prop:monoids}
	The category $\left(\mathrm{Alg}_{\gtimes},\oboxtimes,\pmb{\mathbb{C}}_{\oboxtimes}\right)$  is a monoidal category. If $(A,m_{\gtimes}^{A},\eta_{A})$ and $(B,m_{\gtimes}^{B},\eta_{B})$ are horizontal algebras, then the horizontal product $m_{\gtimes}^{A\oboxtimes B}$ on $A\oboxtimes B$ is defined by :
	\begin{equation}
		m_{\gtimes}^{A\oboxtimes B} = (m_{\gtimes}^{A} \oboxtimes m_{\gtimes}^{B}) \circ R_{A,B,A,B}.
	\end{equation}
	The category $(\mathrm{coAlg}_{\oboxtimes},\gtimes,\pmb{\mathbb{C}}_{\gtimes})$ is a monoidal category. If $(A,\Delta_{A}^{\oboxtimes})$ and $(B,\Delta_{B}^{\oboxtimes})$ are two vertical co-algebras, then
	\begin{equation*}
		\Delta^{\oboxtimes}_{A \gtimes B} = R_{A,B,A,B} \circ ( \Delta^{\oboxtimes}_{A} \gtimes \Delta_{B}^{\oboxtimes})
	\end{equation*}
	defines a coproduct on $A\gtimes B$.
\end{proposition}
\label{ex:endomorphismproperad}

Following \cite{aguiar2010monoidal}, a coalgebra in the category (Alg$_{\gtimes}$,$\oboxtimes, \mathbb{C}_{\oboxtimes}$) is called a \emph{bimonoid}, while an algebra in the same monoidal category is called a \emph{double monoid}.

\begin{definition}(PROS)
	\label{PROS}
	We call a PROS an algebra in the monoidal category (Alg$_{\gtimes}$, $\oboxtimes, \mathbb{C}_{\oboxtimes}$). Otherwise stated a PROS is a tuple $(C,\nabla,m_{\gtimes}^{C},\eta^{C}_{\gtimes},\eta^{C}_{\oboxtimes})$
	with
	$$C\in \mathrm{Coll}_{2},~~\nabla:C\oboxtimes C\rightarrow C,~~m_{\gtimes}^{C}:C\gtimes C\rightarrow C,~~\eta_{\gtimes}^{C}:\pmb{\mathbb{C}}_{\gtimes}\rightarrow C,~~\eta_{\oboxtimes}^{C}:\pmb{\mathbb{C}}_{\oboxtimes}\rightarrow C $$
	with $\nabla$ and $\eta_{\oboxtimes}^{C}$ two horizontal algebra morphisms and $\eta^{C}_{\gtimes}  = \eta_{\oboxtimes}^{C} \circ \eta_{\gtimes}^{\mathbb{C}_{\oboxtimes}}$.
\end{definition}

We proceed with a fundamental example of a set of bicollections endowed with an horizontal and a vertical product. Let $C$ be a collection. First, there exists a canonical isomorphism of bicollections $$\phi: T_{\gtimes} (C \sbt C) \longrightarrow T_{\gtimes}(C)\oboxtimes T_{\gtimes}(C),$$
defined by
\begin{align*}
	 & \phi\bigg(\left[p^{1} \sbt (q^{1}_{1} \gtimes \cdots \gtimes q^{1}_{|p^{1}|})\right] \gtimes \cdots \gtimes  \left[p^{n} \sbt (q^{n}_{1} \gtimes \cdots \gtimes q^{n}_{|p^{n}|})\right] \bigg) = (p^{1} \gtimes \cdots \gtimes p^{n}) \oboxtimes (q^{1}_{1} \gtimes \cdots \gtimes q^{n}_{|p^{n}|}) \\
	 & \phi(1) = 1\oboxtimes 1.
\end{align*}
In fact, from the second assertion of Lemma \ref{lem:lax} with $C_{1}=C_{2}=C$, there exists an isomorphism of bicollections:
\begin{equation*}
	\bigoplus_{n \geq 1}\left(C \oboxtimes T_{\gtimes}(C)\right)^{\gtimes n} \simeq \bigoplus_{n\geq 1} C^{\gtimes n} \oboxtimes T_{\gtimes}(C)
\end{equation*}
with $T^{+}_{\gtimes}(C)(n,m) = T_{\gtimes}(C)$, $n,m\geq 1$ and $T^{+}(C)(0,0) = \{0\}$, since $C_{0} = \{0\}$, we have in fact
\begin{equation*}
	\bigoplus_{n \geq 1}\left(C \oboxtimes T^{+}_{\gtimes}(C)\right)^{\gtimes n} \simeq \bigoplus_{n\geq 1} C^{\gtimes n} \oboxtimes T^{+}_{\gtimes}(C)
\end{equation*}
which implies in turn:
\begin{equation*}
	T_{\gtimes}(C\sbt C) = \mathbb{C}1 \oplus \bigoplus_{n\geq 1} (C\oboxtimes T^{+}_{\gtimes}(C)) \simeq \mathbb{C}(1 \oboxtimes 1) \oplus \bigoplus_{n\geq 1} C^{\gtimes n} \oboxtimes T^{+}_{\gtimes}(C) \simeq T_{\gtimes}(C) \oboxtimes T_{\gtimes}(C).
\end{equation*}
As a consequence, the tensor product $T_{\gtimes}(C) \oboxtimes T_{\gtimes}(C)$ is endowed with an algebra product obtained by pushing forward using $\phi$ the concatenation product on $T_{\gtimes}(P)$.
Assume next that $\mathcal{P}$ is endowed with an operadic composition $\rho : P \sbt P \rightarrow P$. Then $\rho$ induces an \emph{horizontal morphism} (for the concatenation) denoted $T_{\gtimes}(\rho):T_{\gtimes}(P \sbt P)  \rightarrow  T_{\gtimes}(P)$ which equals $\rho$ on $P\sbt P \subset T_{\gtimes}(\mathcal{P}) \oboxtimes T_{\gtimes}(\mathcal{P})$.

\begin{proposition}
	A bicollection $C$ is a coalgebra in the monoidal category $(\mathrm{Alg}_{\gtimes}, \oboxtimes,\pmb{\mathbb{C}}_{\oboxtimes})$ if and only if it is an algebra in the monoidal category
	$(\mathrm{coAlg}_{\oboxtimes}, \gtimes, \pmb{\mathbb{C}}_{\gtimes})$.
\end{proposition}
\begin{proof}
	The two diagrams expressing compatibility between the multiplication map $m_{\gtimes}^{C}$ and the co-product $\Delta_{C}^{\oboxtimes}$ (stating either that $m^{C}_{\gtimes}$ is a $\Delta_{C}^{\oboxtimes}$ morphism or that $ \Delta_{C}^{\oboxtimes}$ is a $m^{C}_{\gtimes}$ algebra morphism) are both equal to the diagram in Fig. \ref{fig:compatmulticoprod}.
	\begin{figure}[!h]
		\begin{tikzcd}
			C \gtimes C \arrow{dd}{m^{\gtimes}} \arrow{r}{\Delta_{C}^{\oboxtimes}\gtimes\Delta_{C}^{\oboxtimes}} & (C\oboxtimes C) \gtimes (C\oboxtimes C) \arrow{d}{R_{C,C,C,C}} \\
			& (C\gtimes C) \oboxtimes (C\gtimes C) \arrow{d}{m^{C}_{\gtimes}\oboxtimes m^{C}_{\gtimes}} \\
			C \arrow{r}{\Delta} & C \oboxtimes C
		\end{tikzcd}
		\caption{\label{fig:compatmulticoprod} Compatibility between the multiplication and comultiplication for $\boxdtimes\gtimes$-bialgebras.}
	\end{figure}
\end{proof}
We use the terminology \raisebox{-1pt}{$\oboxtimes$}$\gtimes$-bialgebras for coalgebra in the monoidal category $(\mathrm{Alg}_{\gtimes},\oboxtimes,\pmb{\mathbb{C}}_{\oboxtimes})$ or for algebra in the category $(\mathrm{coAlg}_{\oboxtimes},\gtimes,\pmb{\mathbb{C}}_{\gtimes})$ (that is for bimonoids). The following definition we be important in the next section, in which we define unshuffle coalgebras and bialgebras in the context of duoidal categories.
\begin{definition}[Co-nilpotent $\boxdtimes\gtimes$-bialgebras]
	A $\boxdtimes\gtimes$-bialgebras $(C,\bar{\Delta},\varepsilon)$ is said  \emph{co-nilpotent} if
	\begin{enumerate}[\indent 1.]
		\item $C(n,n)=\pmb{\mathbb{C}}(n,n)=\mathbb{C}1_{n},~ n \geq 0$ and $C(0,m) = C(m,0) = 0$, $m \geq 1$.
		\item $\bar{\Delta}(c) = \Delta(c) + 1_{m} \oboxtimes c + c \oboxtimes 1_{n},~ c \in C(n,m),~n\neq m$, with
		      \begin{equation*}
			      \bar{\Delta} : C \rightarrow C\oboxtimes C,~ \Delta(C(n,m)) \subset \bigoplus_{\substack{k\geq 0 \\ k\neq n,m}} C(n,k) \gtimes C(k,m),
		      \end{equation*}
		      and $\Delta(1_{n}) = 1_{n} \oboxtimes 1_{n}, n \geq 0$.
		\item $\Delta$ is point-wise nilpotent: for each $c \in C$, there exists an integer $n\geq 0$ such that $\Delta^{n}(c) = 0$.
	\end{enumerate}
\end{definition}
\subsection{\raisebox{-0.75pt}{$\oboxtimes$}$\gtimes$-Hopf algebras and the monoid of horizontal morphisms}
If $A$ is a PROS, we use the notation $\nabla_{A}$ (or $\nabla$ if there are no risks of confusion) instead of $m_{\oboxtimes}^{A}$ for its vertical product. We do the same notational simplifications for \raisebox{-1pt}{$\oboxtimes$}$\gtimes$-bialgebras.

\begin{definition}[\raisebox{-0.75pt}{$\oboxtimes$}$\gtimes$-Hopf algebras]
	A bicollection \raisebox{-0.75pt}{$\oboxtimes$}$\gtimes$-Hopf algebra is a tuple $$(C,\Delta,\nabla,m_{\gtimes}^{C},S,\varepsilon,\eta_{\gtimes}^{C}, \eta_{\oboxtimes}^{C}),$$
	\begin{enumerate}[\indent 1.]
		\item $(C,\Delta,m_{\gtimes}^{C},\varepsilon,\eta_{\gtimes}^{C})$ is a $\oboxtimes$$\gtimes$-bialgebra,
		\item $(C,\nabla,m_{\gtimes}^{C},\eta_{\gtimes}^{C}, \eta_{\oboxtimes}^{C})$ is a PROS,
		\item A morphism $S:C \rightarrow C$ of horizontal algebras such that
		      \begin{align}
			      \label{eqn:antipode}
			       & \nabla \circ (S \oboxtimes \textrm{id}_{C} )\circ \Delta = \nabla \circ (\textrm{id}_{C} \oboxtimes S)  \circ \Delta = \eta \circ \varepsilon.
		      \end{align}
	\end{enumerate}
	The $\boxdtimes\gtimes$-Hopf algebra is said $\emph{connected}$ if $\bigoplus_{n\geq0}C(n,n) \overset{\eta}{\simeq} \pmb{\mathbb{C}}_{\oboxtimes}$.
\end{definition}

\begin{remarque}
	The map $S$ is called an \emph{antipode}. In the definition of a $\boxdtimes\gtimes$-Hopf algebra we do not assume any compatibility conditions between $\Delta$ and $\nabla$. These two morphisms are algebra morphisms with respect to the horizontal algebraic product structure we have on the underlying bicollection, but nothing more. In particular, we can not require for $\Delta$ to be $\nabla$ morphisms, this stems from the fact that $C\oboxtimes C$ is \emph{not} an $\oboxtimes$-algebra, even if $C$ is.

	The map $S$ does not enjoy the same properties as the antipodal map of a plain usual commutative Hopf algebra. In particular, it is not a morphism with respect to the product $\nabla$, nor an anti-comorphism with respect to $\Delta$ nor an unipotent morphism $(S^{2}=\textrm{id})$. We shall see later that in the case of the $\boxdtimes\gtimes$-Hopf algebra canonically associated with the gap-insertion operad, the square of the antipode is in fact a projector.

	Here again a remark on the terminology we use is in order. According to \cite{aguiar2010monoidal}, a \raisebox{-1pt}{$\oboxtimes$}$\gtimes$-Hopf algebra is a bimonoid and a dimonoid endowed with an extra map $S$. Defining the notion of Hopf monoid in a duoidal category is an highly non-trivial task and the various --equivalent-- definitions of an Hopf algebra can lead to different notions of Hopf monoids in a duoidal category. Often, since for a bimonoid (a \raisebox{-1pt}{$\oboxtimes$}$\gtimes$-bialgebra) the comonoidal and the monoidal structures are in different monoidal categories the notion of convolution monoid associated with a bimonoid is meaningless. Therefore, the notion of Hopf monoids in a duoidal category can not be defined using \eqref{eqn:antipode}. See \cite{bohm2013hopf} for a detailed discussion on the different possibilities to define Hopf monoids in a duoidal category.
\end{remarque}

\begin{definition}[Convolution product]
	\label{def:convprod}
	Let $(A, \nabla, m_{\gtimes}^{A}, \eta_{\gtimes}^{A}, \eta_{\oboxtimes}^{A})$ be a PROS and $(B,\Delta,\varepsilon,m_{\gtimes}^{B},\eta_{\gtimes}^{B}, \eta_{\oboxtimes}^{B})$ be a \raisebox{-1pt}{$\oboxtimes$}$\gtimes$ bialgebras. Let $\alpha, \beta : B \rightarrow A$ be two bicollections morphisms. We define the \emph{convolution product} $\alpha \star \beta$ by
	\begin{equation*}
		\alpha \star \beta = \nabla \circ (\alpha \oboxtimes \beta) \circ \Delta.
	\end{equation*}
	The unit for this convolution product $\star$ is $\varepsilon \circ \eta$.
\end{definition}

\begin{proposition}
	%
	%
	%
	The class of horizontal algebra morphisms
	$\mathrm{Hom}_{\mathrm{Alg}_{\gtimes}}(B, A)$ is a monoid for $\star$.

	Besides, assume that $B$ is a \raisebox{-1pt}{$\oboxtimes$}$\gtimes$- Hopf algebra and that $\alpha$ is a PROS morphism $\alpha$. Then $\alpha$ is invertible in the monoid of horizontal algebra morphisms and $\alpha^{-1}=\alpha \circ S$.
\end{proposition}

\begin{proof}
 Since $(\mathrm{Alg}_{\gtimes}, \oboxtimes, \pmb{\mathbb{C}}_{\oboxtimes})$ is a monoidal categories, $\alpha \oboxtimes \beta$ is a $\gtimes$-algebra morphism. Then, $\nabla \circ (\alpha \oboxtimes \beta) \circ \Delta$ is a $\gtimes$-algebra morphism, as a composition of horizontal algebra morphisms. Finally, if $\alpha$ is a PROS morphism, we get
	\begin{align*}
		\nabla \circ ((\alpha \circ S) \oboxtimes \alpha) \circ \Delta
		 & = \nabla \circ (\alpha \oboxtimes \alpha) \circ (S \oboxtimes \textrm{id}) \circ \Delta = \alpha \circ \nabla \circ (S \oboxtimes \textrm{id}) \circ \Delta = \eta \circ \epsilon.
	\end{align*}
\end{proof}

We now show that to the gap-insertion operad $\mathcal{N}\mathcal{C}$ is associated a \raisebox{-0.75pt}{$\oboxtimes$}$\gtimes$-Hopf algebra. As previously explained, the map $\rho_{\mathcal{N}\mathcal{C}}$ extends to an horizontal morphism $T_{\gtimes}(\mathcal{N}\mathcal{C})$ defining a PROS, denoted $\nabla^{T_{\gtimes}(\mathcal{N}\mathcal{C})}$ on $T_{\gtimes}(\mathcal{N}\mathcal{C})$.
The graded dual $ \Delta^{T_{\gtimes}(\mathcal{N}\mathcal{C})}
$ of $\nabla_{T_{\gtimes}(\mathcal{N}\mathcal{C})}$ which reads on a non-crossing partition $\pi$:
\begin{equation}
	\Delta^{T_{\gtimes}(\mathcal{N}\mathcal{C})}(\pi) = \sum_{\substack{\alpha,\beta \in T_{\gtimes}(\mathcal{N}\mathcal{C}) \\ \nabla_{T_{\gtimes}(\mathcal{N}\mathcal{C})}(\alpha \oboxtimes \beta)=\pi}} \alpha \oboxtimes \beta,
\end{equation}
is an horizontal algebra morphism. If $\pi$ is a non-crossing partition, we denote by $\sharp \pi$ the number of non-empty blocks of $\pi$. Define then the algebra morphism $S:T_{\gtimes}(\mathcal{N}\mathcal{C})\rightarrow T_{\gtimes}(\mathcal{N}\mathcal{C})$ by
\begin{equation}
	S(I) = (-1)^{\sharp I}I,\textrm{ if } I\in\textrm{Int} \textrm{ and } S(\pi) = 0 \textrm{ otherwise }.
\end{equation}
Define the counit $\varepsilon : T_{\gtimes}(\mathcal{N}\mathcal{C})\rightarrow \pmb{\mathbb{C}}_{\oboxtimes}$ as the unique horizontal morphism such that $\varepsilon(\{\emptyset\}) = 1_{1}$ and $\varepsilon(\pi) = 0$ otherwise. Define also $\eta : \pmb{\mathbb{C}}_{\oboxtimes} \rightarrow T_{\gtimes}(\mathcal{N}\mathcal{C}) $ by $\eta(1_{n}) = \{\emptyset\}^{n}$ for each integer $n\geq 0$.
\begin{proposition}
	\label{prop:ncunshuffle}
	$(T_{\gtimes}(\mathcal{N}\mathcal{C}), \Delta^{T_{\gtimes}(\mathcal{N}\mathcal{C})}, \nabla_{T_{\gtimes}(\mathcal{N}\mathcal{C})}, S, \varepsilon, \eta)$ is a co-nilpotent \raisebox{-1pt}{$\oboxtimes$}$\gtimes$-Hopf algebra.
\end{proposition}
\begin{proof}
	We check only that $\nabla^{T_{\gtimes}(\mathcal{N}\mathcal{C})} \circ S\oboxtimes \textrm{id} \circ \Delta_{T_{\gtimes}(\mathcal{N}\mathcal{C})} = \nabla_{T_{\gtimes}(\mathcal{N}\mathcal{C})} \circ \textrm{id} \oboxtimes S = \eta \circ \varepsilon$.
	Let $\pi$ be a non-crossing partition. Set $n$ equal to be equal to the number of convex blocks of $\pi$, and let $m$ be the number of blocks of $\pi$ not contained in any other blocks (those blocks are ofently called outer blocks).
	We have first
	\begin{equation*}
		(\nabla_{T_{\gtimes}(\mathcal{N}\mathcal{C})} \circ (\textrm{id} \oboxtimes S) \circ \Delta^{T_{\gtimes}(\mathcal{N}\mathcal{C})}) (\pi) = \sum_{k=0}^{n}\binom{n}{k}(-1)^{k} \pi = 0,
	\end{equation*}
	and also
	\begin{equation*}
		\nabla_{T_{\gtimes}(\mathcal{N}\mathcal{C})} \circ (S \oboxtimes \textrm{id}) \circ \Delta^{T_{\gtimes}(\mathcal{N}\mathcal{C})} = \sum_{k=0}^{m}\binom{m}{k}(-1)^{k}\pi = 0.
	\end{equation*}
\end{proof}


\section{Shuffle point of view on operator-valued probability theory}
\label{sec:shufflepointofview}
The main result of this section is Proposition \ref{prop:shuffle}. We then compute half-shuffle exponentials and show that any (extension as an horizontal morphism of an) operadic morphism on $\mathcal{N}\mathcal{C}$ is a left half-shuffle exponential. We compute the right half-shuffle exponential and the full shuffle exponential.

First, We recall classical results and definitions related to shuffle algebras. The terminology shuffle refers actually to different kind of objects. In the literature, the first meaning to shuffle arises from products of iterated integrals. As such it designates a commutative binary product. The second meaning refers to topological shuffles, the latter being non-commutative. These notions can be traced back at least to the 1950's, when these two notions were axiomatized in the work of Eilerberg--Maclane and Sch\"utzenberger. In this section, shuffle will always refer to the non-commutative case.

A \emph{shuffle} (or \emph{dendrimorphic}) algebra is a $\mathbb{K}$ vector space $D$ together with two bilinear compositions $\prec$ and $\succ$ subject to the following three axioms
\begin{align*}
	 & (a \prec b ) \prec c = a \prec (b \prec c + b \succ c), \\
	 & (a \succ b) \prec c = a \succ (b \prec c),              \\
	 & a \succ (b \succ c) = (a \succ b + a \prec b) \succ c.
\end{align*}
These three relations yield the following associative shuffle algebra product $a \shuffle b = a \prec b + a \succ b$ on $D$. The products $\prec$ and $\succ$ are called, respectively, \emph{left half-shuffle} and \emph{right half-shuffle}.
The standard example of a commutative shuffle algebra (meaning that $a\shuffle b$ = $b \shuffle a$) is provided by the tensor algebra $\bar{T}(V)$ over a $\mathbb{K}$ vector space $V$ endowed with a left half-shuffle recursively defined by
\begin{equation*}
	(x_{1}\otimes \cdots \otimes x_{n}) \prec (y_{1} \otimes  \cdots \otimes y_{n})
	= x_{1} \otimes (x_{2} \otimes \cdots \otimes x_{n} \shuffle y_{1}\otimes\cdots\otimes y_{m}).
\end{equation*}
Shuffle algebras are not naturally unital. This is because it is impossible to split the unit equation $1\shuffle a = a\shuffle 1 = a$, into two equations involving the half-shuffles products $\succ$ and $\prec$. This issue is circumvented by using the "Sch\"utzenberger" trick, that is, for $D$ a shuffle algebra, $\bar{D} = D \oplus \mathbb{K}1$. denotes the shuffle algebra augemented by a unit $\textbf{1}$ such that
\begin{equation*}
	a \prec \textbf{1} = a = \textbf{1} \succ a,~ \textbf{1} \prec a = 0 = a \succ \textbf{1}
\end{equation*}
implying $\textbf{1} \shuffle a = a\shuffle \textbf{1} = a$. By convention, $\textbf{1}\shuffle \textbf{1}= \textbf{1}$, but $\textbf{1} \prec \textbf{1}$ and $\textbf{1} \succ \textbf{1} = 0$ cannot be defined consistently. The following set of left- and right half-shuffle words in $\bar{D}$ are defined recursively for fixed elements $(x_{1},\ldots,x_{n}) \in D$, $n \in \mathbb{N}$
\begin{align*}
	 & w_{\prec}^{(0)}(x_{1},\ldots,x_{n})= 1 = w^{(0)}(x_{1},\ldots,x_{n})               \\
	 & w_{\prec}^{(n)}(x_{1},\ldots,x_{n}) = x_{1} \prec w^{(n-1)}(x_{2},\ldots,x_{n})    \\
	 & w_{\succ}^{(n)}(x_{1},\ldots,x_{n}) = w^{(n-1)}(x_{1},\ldots,x_{n-1}) \succ x_{n}.
\end{align*}
In the case $x_{1} = \cdots = x_{n} = x$, we simply write $x^{\prec n} = w_{\prec}^{(n)}(x,\ldots,x)$ and $x^{\succ n} = w_{\succ}^{(n)}(x,\ldots,x)$.
In the unital algebra $\bar{D}$, both the exponential and logarithm maps are defined in terms of the associative product $\shuffle$:
\begin{equation*}
	\exp_{\shuffle}(x) = 1 + \sum_{n\geq 1}\frac{x^{\shuffle n}}{n!},~ \textrm{log}(1 + x)
	= - \sum_{n\geq 1}(-1)^{n}\frac{x^{\shuffle n}}{n!}.
\end{equation*}
In general, the two sums in the last equation are formal sums. However, in many cases of interest, we are able to identify a subset of elements of $D$ for which these two sums are finite sums.
The half-shuffle exponentials also called "time-ordered" exponentials and are defined by mean of the two shuffles $\prec$ and $\succ$:
\begin{equation*}
	\exp_{\prec}(x) = \textbf{1} + \sum_{n \geq 1} x^{\prec n},~
	\exp_{\succ}(x) = \textbf{1}+\sum_{n\geq 1} x^{n \succ}.
\end{equation*}
Notice that the two \emph{half-shuffle} exponentials are solution of the following fixed point equations:
\begin{equation*}
	X = \textbf{1} + x \prec X,
	\quad
	X = \textbf{1} + X\succ x.
\end{equation*}
These two time-ordered exponentials and the shuffle exponential are the key ingredients to the Hopf algebraic approach of moment-cumulant relations in non-commutative probability theory.

\begin{lemma}[Lemma 2 in \cite{ebrahimi2015cumulants}]
	\label{lem:leftrighinv}
	Let $A$ be a shuffle algebra, and $\bar{A}$ its augmentation by a unit $\textrm{1}$. For $x \in A$, we have
	\begin{equation*}
		\exp_{\succ}(-x) \shuffle \exp_{\prec}(x) = \exp_{\succ}(-x) \shuffle \exp_{\prec}(x) = 1.
	\end{equation*}
\end{lemma}

We proceed with a small overview on the shuffle approach on (scalar-valued) non-commutative probability theory. The core of this approach is developed in \cite{ebrahimi2015cumulants,ebrahimi2016splitting,ebrahimi2018monotone}. Let $(\mathcal{A}, E)$ be a scalar-valued non-commutative probability space. Consider the space $H=\bar{T}(T(\mathcal{A}))$ defined as the linear span of all words on words on elements in $\mathcal{A}$ including the empty word. Then $H$ can be endowed with the unshuffle bialgebra structure $(\Delta,\varepsilon, \Delta_{\prec}, \Delta_{\succ})$, see for example Definition 3 in \cite{ebrahimi2015cumulants}.
Because of the relations satisfied by the half unsuffle coproducts $\Delta_{\prec}$ and $\Delta_{\succ}$, the vector space of all linear forms on $H$ is a shuffle algebra if endowed with the half-shuffles dual to the two unshuffle coproducts.
The authors in \cite{ebrahimi2015cumulants} define a moment morphism $\Phi : H \rightarrow \mathbb{C}$, which is a morphism for the concatenation product on $H$ whose value on a word $a_{1} \otimes \cdots \otimes a_{n} $ (a ``letter'' in $H$) is
\begin{equation}
	\Phi(a_{1}\otimes\cdots \otimes a_{n}) = E(a_{1}\cdot_\mathcal{A} \cdots \cdot_\mathcal{A} a_{n}).
\end{equation}
Then, $\Phi$ is an element of the monoid $\mathrm{Hom}_{\mathrm{Alg}}(H,\mathbb{C})$ of characters of the algebra $H$, endowed with the shuffle product dual to $\Delta$. Since $H$ is connected and nilpotent, $H$ is a Hopf algebra. Therefore $G=\mathrm{Hom}_{\mathrm{Alg}}(H,\mathbb{C})$ is a group and the two half-shuffle exponentials together with the shuffle exponential define three maps from the Lie algebra $Lie(G)$ to $G$. Thus, there exist three linear maps $k,b,m: H \rightarrow \mathbb{C}$ such that:
\begin{equation}
	\label{eqn:momentcumulants}
	\Phi = \varepsilon + k \prec \Phi = \varepsilon + \Phi \succ b = \exp_{\shuffle}(m).
\end{equation}
The three maps $k,b$ and $m$ can be identified with, respectively, the free, boolean and monotone cumulants in the following way. As elements of the Lie algebra $Lie(G)$ they are equal to zero on non-trivial products of words in $H$. On word $w \in T(\mathcal{A})$ they coincide each with one of the tree cumulant functions.
Notice that equation \eqref{eqn:momentcumulants} is equivalent to the free, boolean and monotone moment-cumulant relations. From this perspective cumulants and moments are not on the same footing. Indeed, cumulants are considered infinitesimal objects while moments are encoded by an algebra morphism. Later on, the authors in \cite{ebrahimi2019operads} linked the shuffle approach to a particular operad on non-crossing partitions and the M\"obius inversion to another operad on non-crossing partitions. In this settings, the free cumulants of a random variable become an algebra morphism on the space $N$ of words on non-crossing partitions, seen as solution of a left half-shuffle fixed point equation and the moment-cumulant relations are retrieved through an action (compatible with the convolution coproduct on the dual $N^{\star}$) of an element of the monoid of morphisms on a coalgebra associated with the second operad. To retrieve the moments-cumulants relations for operator-valued probability spaces we will define an operator-valued counterpart of the splitting map defined in \cite{ebrahimi2016splitting}.

\subsection{Unshuffle \raisebox{-1pt}{$\oboxtimes$}$\gtimes$-bialgebras.}

The dual notion of unshuffle algebra appeared after the notion of shuffle algebra in the literature. It has first been considered by L.~Foissy, in its seminal work \cite{foissy2007bidendriform} on the Duchamp--Hivert--Thibon "free Lie algebra" conjecture. We introduce a notion of unshuffle bialgebra adapted to our settings and show that the dual, in a certain sense, of such a bialgebra is a plain shuffle algebra.

\begin{definition}
	\label{def:unshufflecolldeux}
	An unshuffle co-algebra in \textrm{Coll}$_{2}$ is a coaugmented coassociative coalgebra
	$$(\bar{C}= C \oplus \pmb{\mathbb{C}}_{\oboxtimes}, \Delta),~ C(n,n)=0,~n\geq 0$$
	in the monoidal category ($\mathrm{Coll}_{2}$,~$\oboxtimes$, $\pmb{\mathbb{C}}_{\oboxtimes}$) with coproduct
	$$\bar{\Delta}: \bar{C} \rightarrow \bar{C} \oboxtimes \bar{C},~\Delta \in \mathrm{Hom}_{\textrm{Coll}_{2}}(\bar{C},\bar{C}\oboxtimes \bar{C})$$ such that for any $c \in C$,
	$\bar{\Delta}(c) = \Delta(c) + c \oboxtimes 1_{m} + 1_{n} \oboxtimes c$. The reduced coproduct ${\Delta}$ splits into two half unshuffle coproducts $\Delta_{\prec}$ and $\Delta_{\succ}$ such that
	$$
		\Delta = \Delta_{\prec} + \Delta_{\succ}
	$$
	and they satisfy the three following equations:
	\begin{align}
		\label{eqn:unshufflerel}
		 & (\Delta_{\prec} \oboxtimes I)\circ\Delta_{\prec}    = (I \oboxtimes {\Delta}) \circ \Delta_{\succ},~
		({\Delta}\oboxtimes I) \circ \Delta_{\succ}     = (I \oboxtimes \Delta_{\succ}) \circ \Delta_{\succ}                  \\
		 & (\Delta_{\succ} \oboxtimes I) \circ \Delta_{\prec}  = (I \oboxtimes \Delta_{\prec})\circ \Delta_{\succ} \nonumber.
	\end{align}
\end{definition}

In the following definition, we use the shorter notation $\rho^{\oboxtimes}$ for the horizontal algebra product on $C\oboxtimes C$ if $(C,\rho)$ is an horizontal algebra in Coll$_{2}$.
\begin{definition}
	\label{def:unshufflebialgebra}
	\emph{An unshuffle} \raisebox{-1pt}{$\oboxtimes$}$\gtimes$\emph{-bialgebra} is a conilpotent \raisebox{-1pt}{$\oboxtimes$}$\gtimes$-bialgebra $(\bar{C} = C \oplus \pmb{\mathbb{C}}_{\oboxtimes}, \Delta, \rho)$ with
	$$
		\bar{\Delta}(c) = {\Delta}(c) + c \oboxtimes 1_{m} + 1_{n} \oboxtimes c,~c \in C(n,m),
	$$
	and $\bar{\Delta} =  \Delta_{\prec}+\Delta_{\succ}$ is an unshuffle coproduct (see Definition \ref{def:unshufflecolldeux}), satisfying the following compatibility conditions:
	\begin{align}
		 & \Delta \circ \rho = \rho^{\oboxtimes} \circ (\Delta \gtimes \Delta)                                                                                                                                                                                          \\
		 & (\Delta_{\prec}^{+} \circ \rho) (p \gtimes q)= \rho^{\oboxtimes} \circ (\Delta_{\prec}^{+} \gtimes \Delta)(p\gtimes q),~(\Delta_{\succ}^{+} \circ \rho)(p\gtimes q)= \rho^{\oboxtimes} \circ (\Delta_{\succ}^{+} \oboxtimes \Delta)(p\gtimes q)              \\
		 & \nonumber \hspace{5cm} p\not\in \pmb{\mathbb{C}}_{\oboxtimes},~q \in C,                                                                                                                                                                                      \\
		 & \label{eqn:modulco}\Delta_{\prec}^{+}(\rho(1_{m} \gtimes q)) = \rho^{\oboxtimes}((1_{m}\oboxtimes 1_{m}) \gtimes \Delta_{\prec}^{+}(q)), \Delta_{\succ}^{+}(\rho(1_{m}\gtimes q)) = \rho^{\oboxtimes}((1_{m}\oboxtimes 1_{m}) \gtimes \Delta_{\succ}^{+}(q))
	\end{align}
	with $\Delta_{\prec}^{+}(c) = \Delta_{\prec}(c) + c \oboxtimes 1_{n}$, $\Delta_{\succ}^{+}(c) = \Delta_{\succ}(c) + 1_{m} \oboxtimes c$,~$c\in C(m,n)$.
\end{definition}

\subsection{The $\boxdtimes\gtimes$-unshuffle bialgebra of the gap-insertion operad}
\label{ssec:unshuffleGap}

In this section, we focus on non-crossing partitions and define the \raisebox{-1pt}{$\oboxtimes$}$\gtimes$-unshuffle Hopf algebra relevant for the application to operator-valued non-commutative probability theory.

To lighten the notations, we use $\Delta$ and $\nabla$ in place of $\Delta^{T_{\gtimes}(\mathcal{N}\mathcal{C})}$ and $\nabla_{T_{\gtimes}(\mathcal{N}\mathcal{C})}$. We also recall that $\mathrm{Hom}(B)$ is the operad of multilinear maps on $B$ an denote by $\nabla_{\mathrm{Hom}(B)}$ the PROS structure $T_{\gtimes}(\mathcal{N}\mathcal{C})$ induced by the operadic compositon on $\mathrm{Hom}(B)$.

Let $\pi$ be a non-crossing partition of a linearly ordered set $X$. The set of blocks of $\pi$ carries a pre-order defined by declaring for two blocks $V_{1}$ and $V_{2}$ of $\pi$ that $V_{1} \rightarrow_{\pi} V_{2}$ to mean that Conv($V_{2}$) $\cap V_{1} \neq \emptyset$. In plain words, $V_{1} \rightarrow_{\pi} V_{2}$ means that $V_{2}$ is nested in $V_{1}$.

\begin{definition}[Upperset and lowerset]
	\label{def:uppersetlowerset}
	A \emph{lowerset} $L$ of $\pi$ is a set (which may be empty) of blocks of $\pi$ such that if $V \in L$ and $V \rightarrow_{\pi} W$ in $\pi$ then also $W \in L$.
	In plain words, if a block $V$ is in $L$ then all englobing blocks of $V$ are also in $L$ and $L$ is a non-crossing partition.


	An \emph{upperset} of a non-crossing partition $\pi \in \mathrm{N}\mathrm{C}(p)$ is a word $U_{0}\gtimes \cdots\gtimes U_{p}$ of length $p+1$ in $T_{\gtimes}(\mathcal{N}\mathcal{C})$ on non-crossing partitions such that there exists a lowerset $L \in \mathrm{N}\mathrm{C}(p)$ with
	$$\pi = \nabla (L \oboxtimes (U_{1}\gtimes \cdots \gtimes U_{p})).$$
\end{definition}

The notion of upperset and lowerset of a non-crossing partition (and for partitions) can be found in \cite{ebrahimi2019operads}. We denote by ${\sf Lo}(\pi)$ (respectively {\sf Up}($\pi$)) the set of all lowersets (respectively uppersets) of a non-crossing partition $\pi$.

%

Let $\pi$ be a non-empty non-crossing partition. Then a lowerset $L \in \mathrm{N}\mathrm{C}(p)$ of $\pi$ defines an upperset $U_{0}\gtimes\cdots \gtimes U_{p}$. Each of the partitions $U_{i}$ is either equal to the empty partition or is a subset of the partition $\pi$ such that if $V \in U_{i}$ then all blocks $W \in \pi$ such that $V\rightarrow W$ are also in $U_{i}$. Given a lowerset $L$, we denote by $L^{\gtimes}$ the associated upperset, by definition we have:
	\begin{equation}
		\pi = \nabla \left(L\oboxtimes L^{\gtimes}\right).
	\end{equation}
Notice that the lowerset $L$ in the definition of an upperset $U_{1}\gtimes \cdots \gtimes U_{p}$ is unique, the blocks of $L$ are the blocks of $\pi$ not in any of the $U_{i}^{'}s$ and we denote it $U^{\cdot}$.
A cut of $\pi$ is then the data of a lowerset $L$ and an upperset $U$ such that $\pi = \nabla (L \oboxtimes U)$.
Notice that in that case, $L = U^{\cdot}$ and $U = L^{\gtimes}$.
\begin{proposition}
	\label{prop:formulacoproduct}
	Let $\pi$ be a non-empty partition, then
	\begin{equation}
		\Delta_{T_{\gtimes}(\mathcal{N}\mathcal{C})}(\pi) = \sum_{(L,U)\in cut(\pi)} L \oboxtimes U.
	\end{equation}
\end{proposition}

In the following we denote by $T^{+}_{\gtimes}(\mathcal{N}\mathcal{C})$ the subspace of $T_{\gtimes}(\mathcal{N}\mathcal{C})$ generated by words on non-empty partitions. Notice that the horizontal morphism $\Delta$ splits as
\begin{equation*}
	\Delta(w) = \bar{\Delta}(w) + \{\emptyset\}^{m} \oboxtimes w + w \oboxtimes \{\emptyset \}^{n},~
	w \in T_{\gtimes}(\mathcal{N}\mathcal{C})(m,n), m\neq n.
\end{equation*}
In the following definition, we write $1 \in L$ if the block of $\pi$ that contain $1$ is in the lowerset $L$.

To an upperset of a partition corresponds a subset of blocks of $\pi$. Hence, given a cut $(L,U)$ of $\pi$ we write $1 \in U$ (respectively, $1 \in L$) if the blocks of $\pi$ that contains $1$ is in $U$ (in $L$).
\begin{definition}[Half-unshuffles on $T_{\gtimes}(\mathcal{N}\mathcal{C})$]
	\label{def:unshufflenc}
	We define two bicollection maps $\Delta^{+}_{\prec}: T^{+}_{\gtimes}(\mathcal{N}\mathcal{C})\rightarrow T_{\gtimes}(\mathcal{N}\mathcal{C})^{\oboxtimes 2}$, $\Delta^{+}_{\prec}: {T}^{+}_{\gtimes}(\mathcal{N}\mathcal{C})\rightarrow T_{\gtimes}(\mathcal{N}\mathcal{C})^{\oboxtimes 2}$. Let $\pi \in \mathcal{N}\mathcal{C}$ be a non-empty partition and set
	\begin{equation}
		\Delta_{\prec}^{+}(\pi) = \sum_{\substack{(L,U) \in cut(\pi) \\ 1 \in L}}
		L \oboxtimes U,\quad	\Delta^{+}_{\succ}(\pi) = \sum_{\substack{(L,U) \in cut(\pi) \\ 1 \in U}} L \oboxtimes U
	\end{equation}
	We extend $\Delta^{+}_{\prec}$ and $\Delta^{+}_{\succ}$  by setting for a word $w \in T^{+}_{\gtimes}(\mathcal{N}\mathcal{C})$ and a partition $p\in\mathcal{N}\mathcal{C}$ and integer $q\geq 0$:
	\begin{equation}
		\label{eqn:extension}
		\Delta_{\prec}(\{\emptyset\}^{q}p_{1}w) = (\{\emptyset\}^{q} \oboxtimes \{\emptyset\}^{q})\Delta_{\prec}(p)\Delta(w),~\Delta_{\succ}(\{\emptyset\}^{q}pw) = (\{\emptyset\}^{q}\oboxtimes \{\emptyset\}^{q}) \Delta_{\succ}(p)\Delta(w).
	\end{equation}
\end{definition}

\begin{figure}
	\includegraphics[scale=0.75]{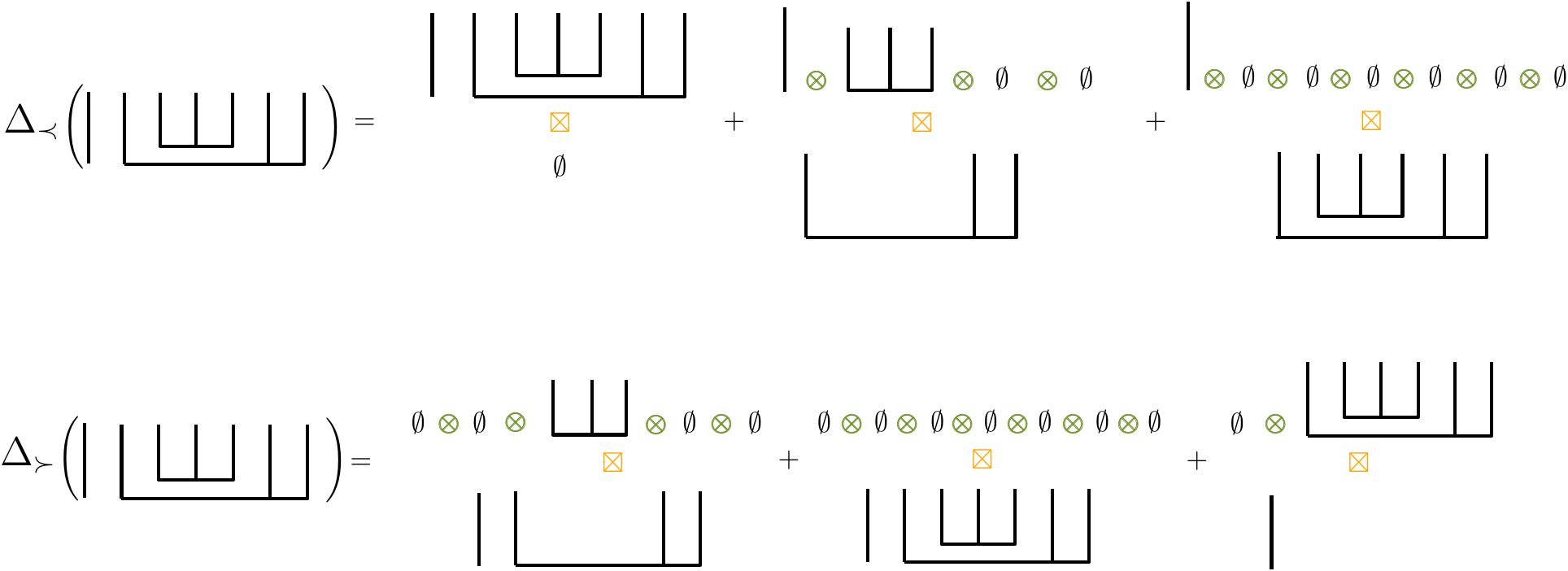}
	\caption{\label{fig:exnoncrossingpartition} The two half unshuffle coproducts acting on a non-crossing partition.}
\end{figure}
From the very definition of the two left/right half-shuffles $\Delta^{+}_{\prec}$ and $\Delta^{+}_{\succ}$, its holds that $\Delta = \Delta_{\prec}+\Delta_{\succ}$.

\begin{proposition}
	\label{prop:unshufflenc}
	$(T_{\gtimes}(\mathcal{N}\mathcal{C}),\Delta,\Delta_{\prec}, \Delta_{\succ})$ is an unshuffle bialgebra in $\mathrm{Coll}_{2}$.
\end{proposition}

\begin{proof}
	For the sake of completeness, we present briefly the arguments given in \cite{ebrahimi2019operads}, Proposition 3.4.3. Let $\pi$ be a partition. It is sufficient to check the relations \eqref{eqn:unshufflerel} for a single partition, because of the equations \eqref{eqn:extension}. Even so, we need to define  lowersets, uppersets and cuts for words on (possibly empty) partitions. By convention, the only cut $(L,U)$ of the empty partition is $(\{\emptyset\}, \{\emptyset\})$. Notice that this convention is compatible with Proposition \ref{prop:formulacoproduct} since $\Delta(\{\emptyset\}) = \{ \emptyset \} \oboxtimes  \{\emptyset \}$

	Let $w = p_{1}\cdots p_{s}$ be a word on partitions with $p_{i} \in \mathcal{N}\mathcal{C}(k_{i})$ with $k_{i} \geq 0$.
	A lowerset of $w$ is a word $L_{1} \cdots L_{p}$ with $L_{i}$ a lowerset of the partition $p_{i}$. The notion of an upperset for $w$ is defined similarly, an upperset of $W$ is a word on uppersets one for each of the partition $p_{i}$. The notion of cut for partitions is then downwardly transferred to words on partitions. Then we have the formulas:
	\begin{equation}
		\Delta_{\prec}^{+}(w) = \sum_{\substack{(L,U) \in cut(w) \\ 1 \in L}}
		L \oboxtimes U,\quad	\Delta^{+}_{\succ}(w) = \sum_{\substack{(L,U) \in cut(w) \\ 1 \in U}} L \oboxtimes U
	\end{equation}
	for a word $w \in T_{\gtimes}^{+}(\mathcal{N}\mathcal{C})$.
	We say that $(L,M,U)$ is a compatible pair of cuts of $w$ if $L$ is a lowerset of $w$, $U$ is an upperset of $w$ with $L^{\gtimes} = \nabla U \oboxtimes M$ and $U^{\cdot}=\nabla(L \oboxtimes M$) (because $\Delta_{T_{\gtimes}(\mathcal{N}\mathcal{C})}$ is coassociative these two conditions are equivalent) with $L, M, U \not\in \pmb{\mathbb{C}}_{\oboxtimes}$. We denote by $cut_{2}(w)$ the set of compatible pairs of cuts of a words in $T_{\gtimes}(\mathcal{N}\mathcal{C})$. Let $\pi$ be a non-crossing partition, we have:

	\begin{align*}
		 & (\Delta_{\prec} \oboxtimes \textrm{id}) \circ \Delta_{\prec} (\pi) = \sum_{\substack{(L,M,U)\in cut_{2}(\pi) \\ 1 \in L}} L \oboxtimes M^{} \oboxtimes U^{} = (\textrm{id} \oboxtimes \Delta) \circ \Delta_{\prec} (\pi),\\
		 & (\Delta_{\succ} \oboxtimes \textrm{id}) \circ \Delta_{\prec} (\pi) = \sum_{\substack{(L,M,U)\in cut_{2}(\pi) \\ 1 \in M}} L \oboxtimes M^{} \oboxtimes U^{} = (\textrm{id} \oboxtimes \Delta_{\prec}) \circ \Delta_{\succ} (\pi)\\
		 & (\Delta \oboxtimes \textrm{id}) \circ \Delta_{\succ} (\pi)= \sum_{\substack{(L,M,U)\in cut_{2}(\pi)          \\ 1 \in U}} L \oboxtimes M^{} \oboxtimes U^{} = (\textrm{id} \oboxtimes \Delta_{\succ}) \circ \Delta_{\succ} (\pi).
	\end{align*}
\end{proof}
Thanks to \raisebox{-1pt}{$\oboxtimes$}$\gtimes$-bialgebra $T_{\gtimes}(\mathcal{N}\mathcal{C})$ being conilpotent, the following proposition holds.
\begin{proposition}
	The $\raisebox{-1pt}{$\oboxtimes$}\gtimes$-bialgebra $(T_{\gtimes}(\mathcal{N}\mathcal{C}), \Delta, \Delta_{\prec}, \Delta_{\succ})$ endowed with the vertical product $\nabla$ is a unshuffle Hopf algebra.
\end{proposition}
The splitting of the horizontal morphism $\Delta_{T_{\gtimes}(\mathcal{N}\mathcal{C})}$ into the two half-unshuffle $\Delta_{\prec}$ and $\Delta_{\succ}$ induces two bilinear (non-associative) composition on the vector space of bicollection morpisms from $T_{\gtimes}^{+}(\mathcal{N}\mathcal{C})$ to $T^{+}_{\gtimes}(\mathrm{Hom}(B))$ (with obvious notations):
\begin{equation*}
	f \prec g = \nabla_{\mathrm{Hom}(B)}\circ (f \oboxtimes g) \circ \Delta_{\prec},~ f \succ g = \nabla_{\mathrm{Hom}(B)}\circ(f \oboxtimes g) \circ \Delta_{\succ}, f,g \in \mathrm{Hom}_{\mathrm{Coll_{2}}}(T_{\gtimes}(\mathcal{N}\mathcal{C}),~T_{\gtimes}(\mathrm{Hom}(V))).
\end{equation*}
Recall that we wenote by $\eta_{\mathrm{Hom}(B)} : \pmb{\mathbb{C}}_{\oboxtimes} \rightarrow T_{\gtimes}(\mathrm{Hom}(V))$ the unique horizontal morphism such that $\eta_{\mathrm{Hom}(V)}(1_{1}) = \textrm{id}_{V}$.
We set $$f \prec \left(\eta_{\mathrm{Hom}(B)}\circ \varepsilon\right) = \left(\eta_{\mathrm{Hom}(V)} \circ \varepsilon \right)\succ f = f$$
and $$ \left(\eta_{\mathrm{Hom}(B)} \circ \varepsilon \right) \prec f = f \succ \left(\eta_{\mathrm{Hom}(B)} \circ \varepsilon\right) = 0.$$
The following proposition is a corollary of Proposition \ref{prop:unshufflenc} and equations \eqref{eqn:unshufflerel}. With the notation $\overline{\mathrm{Hom}}(T_{\gtimes}(\mathcal{N}\mathcal{C}), T_{\gtimes}(\mathrm{Hom}(B)) = \mathbb{C} (\eta \circ \varepsilon_{\mathrm{Hom}(B)}) \oplus \mathrm{Hom}(T^{+}_{\gtimes}(\mathcal{N}\mathcal{C}), T_{\gtimes}(\mathrm{Hom}(B))$,
the following proposition is a direct corollary of the last proposition.
\begin{proposition}
	\label{prop:shuffle}
	$(\overline{\mathrm{Hom}}(T_{\gtimes}(\mathcal{N}\mathcal{C}), T_{\gtimes}(\mathrm{Hom}(B))), \prec,\succ,\star)$ is a shuffle algebra.
\end{proposition}

\begin{definition}
	\label{def:infinitesimalmorphism}
	An \emph{infinitesimal morphism} $k: T_{\gtimes}(\mathcal{N}\mathcal{C}) \rightarrow T_{\gtimes}(\mathrm{Hom}(B))$ is a bicollections morphism such that
	\begin{equation*}
			k(\{\emptyset\}^{p} \otimes \pi \otimes \{\emptyset\}^{q}) = \mathrm{id}_{B}^{p} \otimes k(\pi) \otimes \mathrm{id}_{B}^{q},~\pi \in \mathrm{NC},~\pi \neq \{\emptyset\},
	\end{equation*}
and equal to zero on elements in $T_{\gtimes}(\mathcal{N}\mathcal{C})$ not of the form above.
\end{definition}
In the following three section we compute the left half-shuffle, the right half-shuffle exponential and the shuffle exponential of a infinitsemal character.
\subsection{Half-shuffle and shuffle exponentials}
\label{sec:halfshuffle}
In this section we compute the half-shuffle and shuffle exponentials of infinitesimal morphisms. Those exponentials are always horizontal algebra morphisms and are compatible with the gap-insertion composition under some hypothesis. The three main results of this section are contained in Proposition \ref{prop:leftshuffle}, \ref{prop:rightshuffle} and \ref{prop:monotonemorphism}.
\subsubsection{Left half-shuffle}
Given an infinitesimal character $k$, we compute the half-shuffle exponential $K=\exp_{\prec}(k)$.
Recall that $K$ is the solution of the following fixed point equation
\begin{equation}
	\label{eqn:fixedpointfree}
	K =  \eta_{\mathrm{Hom}(B)} \circ \varepsilon  + k \prec K,
\end{equation}
and that we put $1_{n}$ for the partition in $\mathrm{NC}(n-1)$ with only one block, $n\geq 2$. If $\pi$ is partition, recall that $\sharp \pi$ denotes the number of blocks of $\pi$.
\begin{proposition}
	\label{prop:leftshuffle}
	With the notation introduced so far the left half-shuffle exponential $\exp_{\prec}(k)$ of an infinitesimal morphism $k$ is an horizontal morphism. Beside $K$ is a PROS morphism if and only if
	\begin{equation}
		k(1_{n})\circ_{n}k(1_{m})=k(1_{m})\circ_{1}k(1_{n})
	\end{equation}
	and $k(\pi)=0 \textrm{ if } \sharp \pi > 1$.
\end{proposition}

\begin{proof}
	Set $K|_{1}$ to be the restriction of $K$ to $\mathcal{N}\mathcal{C}$. We show that the solution $K$ of \eqref{eqn:fixedpointfree} is an horizontal morphism. We do it recursively. Let $\tilde{K}$ be the horizontal morphism extending the values of $K_{|1}$. The two maps $K$ and $\tilde{K}$ agree on words on partitions with  no non-empty blocks, since in that case $K(\{\emptyset\}^{q}) = \tilde{K}(\{\emptyset\}^{q}) = \left(T_{\gtimes}(\eta_{\mathrm{Hom}(b)})\circ \varepsilon\right) (\{\emptyset^{q}\})$. Assume next that $K$ and $\tilde{K}$ agree on words of partitions with a total number of non-empty blocks at most equal to $N\geq 1$. Pick a word on partitions with $N+1$ blocks and write $w=\emptyset^{p}\pi\tilde{w}$, with $\pi \neq \{\emptyset\}$ and $|\tilde{w}|$ a word of length $s$. Let $V$ be the block of the partition associated with $\pi$ that contains $1$. Then by definition of an infinitesimal character, we get
	\begin{equation*}
		K(w) = (k \prec K)(w)
		=  \sum_{\substack{(L,U) \in cut(\pi)\\ 1\in L}} \nabla_{\mathrm{Hom}(B)} \bigg(\{\emptyset\}^{p} k(L)\{\emptyset\}^{s}\oboxtimes K(\{\emptyset\}^{p}U \tilde{w})\bigg).
	\end{equation*}
	Since the number of non-empty blocks of $U$ and $U\tilde{w}$ is less than the number of non-empty blocks of $w$, we get
	\begin{equation*}
		K(w) = \sum_{\substack{(L,U) \in Cut(\pi) \\ 1 \in L}} \textrm{id}_{B}\,\nabla_{\mathrm{Hom}(B)}\bigg(k(L)\oboxtimes \tilde{K}(U)\bigg)\,\tilde{K}(\tilde{w})
		= \textrm{id}_{B}(k\prec K)(\pi)\tilde{K}(\tilde{w}) = \tilde{K}(w).
	\end{equation*}

	Next, Assume $k(\pi)=0 $ if $\sharp \pi > 1$ and $k(1_{n})\circ_{n}k(1_{m})=k(1_{m})\circ_{1}k(1_{n})$. Let $\phi:\mathcal{NC} \rightarrow \mathrm{Hom}(B)$ be the operadic morphism extending the values $k(1_{n}),~n\geq 1$. If $\pi$ is a partition with only one block, then
	\begin{equation}
		K(\pi)
		=  (\eta_{\mathrm{Hom}(B)} \circ \varepsilon)(\pi) + (k \prec K)(\pi)
		= 0 + k(\pi) \circ (K(\emptyset^{|\pi|})) = k(\pi)=\phi(\pi).
	\end{equation}
	Assume that the result holds for words on partitions with at most $N$ blocks, $K(\pi_{1}\cdots\pi_{p}) = \phi(\pi_{1}\cdots\pi_{p})$ for every element $\pi_{1}\cdots\pi_{p} \in T_{\gtimes}(\mathcal{N}\mathcal{C})$ with $\sharp{\pi_{1}} + \cdots + \sharp{\pi_{p}} \leq N$.
	Let $\pi$ be a partition with $N+1$ blocks. We denote by $V$ the block of $\pi$ that contains $1$. With this notation, we have
	\begin{align*}
		K(\pi) = (\eta_{\mathrm{Hom}(B)} \circ \varepsilon)(\pi) + k \prec K (\pi)
		       =\sum_{(L,U) \in Cut(\pi)}\nabla_{\mathrm{Hom}(B)}(k(L) \oboxtimes K(U))
		       &= \nabla_{\mathrm{Hom}(B)}(k(1_{\sharp V}) \oboxtimes \phi(U)) \\ &= \phi(\pi).
	\end{align*}
	The last equality follows by application of the recursive hypothesis since $\sharp U \leq N$. Now assume that the solution $K_{|1}$ is an operadic morphism. Let $\pi \neq \emptyset$ be a non-crossing partition.
	\begin{align*}
		K(\pi) & = K(1_{\sharp V}) \circ (K(\pi_{0}),\ldots,K(\pi_{|V|})) = (k \prec K)(\pi)                                                                                                       \\
		       & = \sum_{(L,U) \in Cut(\pi)}\nabla_{\mathrm{Hom}(B)}(k(L) \oboxtimes K(U))(\pi)                                                                                                    \\
		       & =\nabla_{\mathrm{Hom}(B)}(k(1_{\sharp V}) \oboxtimes (K(\pi_{1}) \gtimes \cdots \gtimes K(\pi_{|V|}))) + \sum_{ L \neq \{V\} }\nabla_{\mathrm{Hom}(B)}(k(L) \oboxtimes K(U))(\pi).
	\end{align*}
	This last equality implies $\sum_{ L \neq \{V\} }\nabla_{\mathrm{Hom}(B)}(k(L) \oboxtimes K(U))(\pi) = 0$. A simple recursive argument on the number of blocks ends the proof.
\end{proof}

\subsubsection{Right half-shuffle exponential}
Given an infinitesimal character $b: T_{\gtimes}(\mathcal{N}\mathcal{C}) \rightarrow T_{\gtimes}(\mathrm{Hom}(B))$, we compute the right half-shuffle exponential $\exp_{\succ}(b)$
Recall that $\exp_{\succ}(b)$ is the unique solution of the following fixed point equation:
\begin{equation}
	\label{eqn:fixedpointfree2}
	B = \eta_{\mathrm{Hom}(B)} \circ \varepsilon +  B \succ b.
\end{equation}
Let $\pi$ be a non-crossing partition. The \emph{adjacency forest} $\tau(\pi)$ of $\pi$ encodes nesting of the blocks of $\pi$. To each block of $\pi$ we associate a vertex. Two blocks are connected in $\tau(\pi)$ if the convex hull of one of the block contains the other block. The root of each tree in $\tau(\pi)$ is a block not contained in any other block. In particular, the adjacencey forest of an irreductible partition (see \cite{arizmendi2015relations}) is a tree.

We say that an horizontal morphism $B:T_{\gtimes}(\mathcal{N}\mathcal{C})\rightarrow T_{\gtimes}(\mathrm{Hom}(B))$ is \emph{boolean} if it is equal to zero on any non-crossing partitions with at least two nested blocks. Those partitions have an adjacency forest with at least one tree containing two vertices. In addition, if $I_{1}\cdots I_{p}$ is an interval partition, we require that:
\begin{equation}
	B(I_{1}\cdots I_{p}) = \Big(\cdots \Big(B(I_{p})\circ_{1}B(I_{p-1})\cdots \Big) \circ_{1}  B(I_{2})\Big)\circ_{1} B(I_{1})\Big).
\end{equation}

\begin{proposition}
	\label{prop:rightshuffle}
	With the notation introduced so far, the bicollection morphism $B$ solution of the fixed point equation \eqref{eqn:fixedpointfree2} is a horizontal morphism. Besides, $B$ is boolean if and only if
	$$b(1_{n})\circ_{n}b(1_{m}) = b(1_{m})\circ_{1}b(1_{n})$$ and $ b(\pi)=0 \textrm{ if } \sharp \pi > 1$.
\end{proposition}

\begin{proof}
	The proof is very similar to the free case. Let $\tilde{B}$ be the horizontal morphism extending the values of $\overline{b}$. The two maps $B$ and $\tilde{B}$ agree on words on partitions with at most $1$ non-empty block. Assume that $B$ and $\tilde{B}$ agree on words on partitions with at most $N$ non-empty blocks.

	Pick a word on partitions with $N+1$ blocks and write $w=\emptyset^{p}\pi\tilde{w}$, with $\pi \neq \{\emptyset\}$ and $\tilde{w}$ a word of length $s$. Let $V$ be the block of the partition associated with $\pi$ that contains $1$. Then by definition of an infinitesimal morphism, we get
	\begin{equation*}
		B(w) = (\underline{k} \prec K)(w) =  \sum_{\substack{(L,U) \in Cut(\pi)\\ 1\not\in L}} \nabla_{\mathrm{Hom}(b)}\bigg(B(\{\emptyset\}^{p} L\tilde{w})\oboxtimes b(\{\emptyset\}^{p}U\{\emptyset\}^{|\tilde{w}|}) \bigg).
	\end{equation*}

	Since the number of non-empty blocks of $L\tilde{w}$ and $L$ is less than the number of non-empty blocks in $w$, we get:
	\begin{equation*}
		B(w) = \sum_{\substack{(L,U) \in Cut(\pi)\\ 1\not\in L}} \textrm{id}_{B}^{p}\nabla_{\mathrm{Hom}(b)}\bigg(\tilde{B}(L)\oboxtimes b(U) \bigg) \tilde{B}(\tilde{w})
		=\textrm{id}_{B}^{p}(B \succ b)(\pi) \tilde{B}(\tilde{w})= \tilde{B}(w).
	\end{equation*}

	We assume that $b(\pi)=0 $ if $\sharp \pi > 1$. Let $\phi$ be the boolean morphism that extends the values $b(1_{n})$, $n \geq 1$. We show recursively on the total number of non-empty blocks of word on partitions in $T_{\gtimes}(\mathcal{N}\mathcal{C})$ that $B = T_{\gtimes}(\phi)$. First, the two maps coincide on words on partitions with a total number of non-empty blocks less than one. Let $N \geq 1$ and assume that $T_{\gtimes}(\phi)$ and $B$ are equal on multi-partition with at most $N$ blocks. Pick $\pi$ a partition with $N+1$ blocks. Assume first that the \emph{adjacency forest} of $\pi$ contains at least one tree not equal to the root.
	\begin{align}
		\label{prob:booleanexpo::proof::equation1}
		B(\pi) = (B \succ b)(\pi)
		= \sum_{ 1 \not\in L \in \lmss{Lo}(\pi)} \nabla_{\mathrm{Hom}(B)}({B}(L) \oboxtimes b(U))
	\end{align}

	A cut of the partition $\pi$ corresponds to an admissible cut of its adjacency tree. Since $\overline{b}(U) = 0$ if $U$ is a word on partitions either containing at least two non-empty partitions or equal to some $\emptyset^{p}, p \geq 1$, the cuts that contribute to the sum on the righthand side of \eqref{prob:booleanexpo::proof::equation1} extract one and only one leaf of the adjacency forest. Hence, if the block $V$ of $\pi$ that contains $1$ contains at least another block in its convex hull, $B(\pi)=0$. Assume the opposite. It implies that the partition $\pi \backslash V$ is not an interval partition (and is not empty). Besides,
	$$
		B(\pi) = \nabla_{\mathrm{Hom}(b)}(b(V) \oboxtimes B(\emptyset \gtimes\pi \backslash V))
	$$
	The induction hypothesis implies $B(\emptyset \gtimes\pi \backslash V)) = 0$. Now suppose that $\pi = I_{1} \cdots I_{p}$ is an interval partition.
	\begin{equation*}
		B(\pi) = \nabla_{\mathrm{Hom}(b)}(b(I_{1}) \oboxtimes B(I_{2} \cdots I_{p})).
	\end{equation*}
	We apply the recursive hypothesis on $B(I_{2} \cdots I_{p})$ to end the proof.
\end{proof}
\subsubsection{Shuffle exponential}
\label{ssec:shuffleexp}

In this section we compute the shuffle exponential \eqref{eqn:shuffleexp}. The restriction of this horizontal morphism to non-crossing partitions (operators with one output) is not compatible in any way, to the extent of our knowledge with the operation of gap-insertion. This boils down to the fact that the tree factorial defined hereafter is not multiplicative.

\begin{definition}[Monotone partition]
	Let $\pi$ a partition with $k$ blocks. An admissible labelling of the blocks by integers in $\llbracket 1,k\rrbracket$ is an injective labelling which is increasing with respect to the nesting preoder on the blocks: If a block $V \in \pi$ is contained in the convex hull of a block $W$ in $\pi$ then the label of $V$ is less than the label of $W$. A partition with an admissible labelling of its blocks is called a monotone partition. The set of all monotone partitions is denoted $\mathrm{NC}_{m}$.
\end{definition}

\begin{definition}[Tree factorial, \cite{arizmendi2015relations}, Definition 3.2]
	The tree factorial $t!$ of a rooted tree $t$ is recursively defined as follows. Let $t$ be a rooted tree with $n > 0$ vertices. It $t$ consists of a single vertex, set $t!=1$. Otherwise $t$ can be decomposed into its root vertex and branches $t_{1},\ldots,t_{r}$ and we defined recursively the number
	\begin{equation*}
		t! = n \cdot t_{1}!\cdots t_{k}!
	\end{equation*}
	The tree factorial of a forest is the product of the factorials of the constituting trees.
\end{definition}

\begin{proposition}[\cite{arizmendi2015relations}, Proposition 3.3]
	The number $m(\pi)$ of monotone labellings of a non-crossing partition depends only on its adjacency forest $\tau(\pi)$ and is given by $m(\pi)=\frac{\sharp \pi !}{\tau(\pi)!} $
\end{proposition}

Let $m: T_{\gtimes}(\mathcal{N}\mathcal{C})\rightarrow T_{\gtimes}(\mathrm{Hom}(B))$ be an infinitesimal morphism and define the shuffle exponential by
\begin{equation}
	\label{eqn:shuffleexp}
	\exp_{\star}(m) = \eta_{\mathrm{Hom}(B)} \circ \varepsilon + \sum_{p\geq 1} \frac{1}{p!}m^{\star p}.
\end{equation}

\begin{proposition}
	\label{prop:monotonemorphism}
	Pick $m:T_{\gtimes}(\mathcal{N}\mathcal{C})\rightarrow T_{\gtimes}(\mathcal{N}\mathcal{C})$ an infinitesimal morphism such that:
	\begin{equation}
		m(1_{n}) \circ_{1} m(1_{m}) = m(1_{m}) \circ_{m+1} m(1_{n})
	\end{equation}
	with $m(\pi) = 0$ if $\sharp \pi > 1$. Then, $\exp_{\star}$ is an horizontal morphism and
	\begin{equation*}
		\exp_{\star}(m)(\pi)
		= \frac{1}{\tau(\pi)!} \exp_{\prec}(m)(\pi),~\pi \in \mathcal{N}\mathcal{C}.
	\end{equation*}
\end{proposition}

\begin{proof}
	Let $\pi$ be a non-crossing partition with $k$ blocks.
	The number of admissible labelings of the partition $\pi$ is equal to $\frac{k!}{\tau(\pi)!}$. Hence, to prove the statement, it is sufficient to show that
	\begin{equation*}
		\exp_{\star}(m)(\pi)
		= \frac{1}{k!} \sum_{\pi \in \mathcal{N}\mathcal{C}_{m}} \exp_{\prec}(m)(\pi).
	\end{equation*}
	To that end, we show first that there exists a natural embedding of the set of admissible labelings of a partition into the set of multiple admissible cuts of a partition. A multiple cuts of a partition $\pi$ is a sequence $\left(L_{1},\ldots,L_{s}\right)$ of (possibly empty) subsets of blocks of $\pi$ such that $L_{i}$ is a lower cut of $L_{i-1}$ with the convention $L_{0} = \pi$.
	For such a multiple cut of $\pi$, we denote by $L_{i} \backslash L_{i-1} $ the words on partition in $T_{\gtimes}(\mathcal{N}\mathcal{C})$ such that
	\begin{equation*}
		\nabla(L_{i} \oboxtimes (L_{i-1}\backslash L_{i})) = L_{i-1}.
	\end{equation*}
	Let $(\pi,\ell)$ be a monotone partition.  We associate to the labelling $\ell$ of the block a multiple cut $\lmss{L}(\pi,\ell)$ of $\pi$ as follows. For each integer $i \in \llbracket 1,k \rrbracket$, we denote by $V_{i}$ the block of $\pi$ labelled with the integer $i$. We define recursively $\lmss{L}(\pi,\ell)$ by the following rule:
	\begin{equation*}
		\lmss{L}(\pi,\ell)_{0}=\pi,~\lmss{L}(\pi,\ell)_{i} = \lmss{L}(\pi,\ell)_{i-1} \backslash V_{i}.
	\end{equation*}
	Because the labelling $\ell$ is monotone, we obtain indeed a multiple cut of $\pi$. Next, from the definition of the coproduct $\Delta$, we see that:
	\begin{equation*}
		\exp_{\star}(m)(\pi) = \sum_{s\geq1}\sum_{(L_{1},\ldots,L_{s})} \frac{1}{s!} \nabla_{\mathrm{Hom}(B)}^{\oboxtimes s}(m(L_{s-1} \backslash L_{s}) \oboxtimes \cdots \oboxtimes m(L_{0} \backslash L_{1})),
	\end{equation*}
	with $\nabla_{\mathrm{Hom}(B)}^{\oboxtimes s}$ defined recursively by $\nabla_{\mathrm{Hom}(B)}^{\oboxtimes 1} = \nabla_{\mathrm{Hom}(B)}$ and $\nabla_{\mathrm{Hom}(B)}^{\oboxtimes (s+1)}=\nabla_{\mathrm{Hom}(B)}^{\oboxtimes s} \oboxtimes \textrm{id}\circ \nabla_{\mathrm{Hom}(B)} $. From the definition of an infinitesimal character, the sum on the right hand side of the last equation reduces to
	\begin{equation*}
		\exp_{\star}(m)(\pi) = \sum_{(\pi,\ell) \in \mathcal{N}\mathcal{C}_{m}}\frac{1}{k!} \nabla_{\mathrm{Hom}(B)}^{\oboxtimes s}(m(\lmss{L}(\pi,\ell)_{s-1} \backslash \lmss{L}(\pi,\ell)_{s}) \oboxtimes \cdots \oboxtimes m(\lmss{L}(\pi,\ell)_{0} \backslash \lmss{L}(\pi,\ell)_{1})).
	\end{equation*}
	The result follows from the last equation.
\end{proof}

\section{Operator-valued moment-cumulant relations}
\subsection{Operad of words insertions}
\label{sec:splitting}
In this section, we introduce an \raisebox{-1pt}{$\oboxtimes$}$\gtimes$ unshuffle Hopf algebra (see Definition \ref{def:unshufflebialgebra}) associated with an operad of words on random variables (defined hereafter). We proceed with the definition of a \emph{splitting map} from this unshuffle Hopf algebra to the unshuffle Hopf algebra of words on non-crossing partitions we defined in the previous section. This map has already been defined in \cite{ebrahimi2016splitting} for the scalar case. We adapt the arguments in \cite{ebrahimi2016splitting} to our setting to show that the dual of the splitting map induces a morphism between shuffle algebras.
We prove finally that the operator-valued moment-cumulant relations for free and boolean cumulants are equivalent to two half-shuffle fixed point equations, see Proposition \ref{prop:momentcumulantsrelation}.

In this section, all non-crossing partitions have their legs coloured with elements in the algebra $\mathcal{A}$. We use the same notation NC for the set of all coloured non-crossing partitions. The material exposed in the previous sections extends readily to coloured non-crossing partitions. A generic coloured non-crossing partition is written
$$
	\pi \otimes a_{1}\otimes \cdots \otimes a_{p},~\pi \in \mathrm{NC}(p),~a_{i} \in \mathcal{A}.
$$
We give only sketches of the proofs, if any, and the reader is directed to \cite{ebrahimi2016splitting} where he or she will find detailed proofs readily adapted to our settings
 For the remainder of the section, we come to our (heavier) notations $\Delta_{\prec}^{T_{\gtimes}(\mathcal{N}\mathcal{C})}$, $\Delta_{\succ}^{T_{\gtimes}(\mathcal{N}\mathcal{C})}$ and $\varepsilon^{T_{\gtimes}(\mathcal{N}\mathcal{C})}$ for the unshuffle structure on
$T_{\gtimes}(\mathcal{N}\mathcal{C})$.
We start with the definition of the operad of words insertions. We denote by $T(\mathcal{A})$ the vector space of all non-commutative polynomials on elements in the algebra $\mathcal{A}$,
$$T(\mathcal{A}) = \bigoplus_{n\geq 1} \mathcal{A}^{\otimes n}$$
We augment this space with the empty word $\emptyset$ and set $\bar{T}(\mathcal{A}) = \mathbb{C}\emptyset \oplus T(\mathcal{A})$. A word $w_{1}\cdots w_{p} \in T(\mathcal{A})$ is graded by its length plus one: $$\lmss{i}(w_{1}\cdots w_{p}) = p+1.$$ The empty word has length $0$.

\begin{definition}[Words insertions operad]
To the space $\bar{T}(\mathcal{A})$ is associated the collection $$\bar{T}(\mathcal{A})(n) = \mathcal{A}^{\otimes (n-1)},~n\geq 1$$
The words insertions operadic law $\rho_{\mathcal{W}\mathcal{I}}$ is defined by:
\begin{equation}
	\begin{array}{lccc}
		\rho_{\mathcal{W}\mathcal{I}}: & \bar{T}(\mathcal{A})\sbt\bar{T}(\mathcal{A})                  & \longrightarrow & \bar{T}(\mathcal{A})                    \\
		                               & x_{1}\cdots x_{p} \otimes (y_{1}\otimes\ldots\otimes y_{p+1}) & \mapsto         & y_{1}x_{1}y_{2}x_{2}\cdots x_{p}y_{p+1}
	\end{array}
\end{equation}
The empty word $\emptyset$ acts as the unit for the \emph{word insertion operad}.
\end{definition}
We denote by $\mathcal{W} = T_{\gtimes}(\bar{T}(\mathcal{A}))$ the space of all words on elements of $\bar{T}(\mathcal{A})$, augmented with an element $1$ with $0$ inputs and outputs. We denote by $\nabla_{\mathcal{W}}:\mathcal{W}\oboxtimes\mathcal{W}\rightarrow\mathcal{W}$ the PROS product induced by $\rho_{\mathcal{W}\mathcal{I}}$.
Finally, we set $\Delta^{\mathcal{W}}: \mathcal{W}\rightarrow \mathcal{W}\oboxtimes \mathcal{W}$ the unique horizontal algebra morphism such that:
\begin{equation*}
	\Delta^{\mathcal{W}}(w) = \sum_{\substack{\alpha,\beta \in \mathcal{W}, \\ w = \nabla_{\mathcal{W}}(\alpha \oboxtimes \beta)}} \alpha \oboxtimes \beta,~~ w \in \bar{T}(\mathcal{A}).
\end{equation*}
Owing to associativity of $\nabla_{\mathcal{W}}$, the map $\Delta^{\mathcal{W}}$ is a vertical coproduct with counit $\epsilon^{\mathcal{W}}: \mathcal{W} \rightarrow \mathbb{C}_{\oboxtimes}$:
$$
	\epsilon^{\mathcal{W}}(w) = \delta_{w = \emptyset^{n}} 1_{n},~ w \in \mathcal{W}(n,m).
$$
Notice that $(\mathcal{W}, \Delta^{\mathcal{W}}, \varepsilon^{\mathcal{W}})$ is a conilpotent \raisebox{-1pt}{$\oboxtimes$}$\gtimes$-bialgebra, since $\mathcal{W}(n,n) = \mathbb{C},~n\geq1$ and $\Delta^{\mathcal{W}}$ splits as
\begin{equation*}
	\Delta^{\mathcal{W}}(w) = \{\emptyset\}^{n}\oboxtimes w + w \oboxtimes \{\emptyset \}^{m} + \bar{\Delta}^{\mathcal{W}}(w),~ w \in \mathcal{W}(n,m),~ n\neq m
\end{equation*}
with $(\bar{\Delta}^{\mathcal{W}})^{n}(w)=0$ if $n \geq |w|$.
We now proceed with a similar construction we gave for the operad of non-crossing partitions. Let $w$ a word in $\mathcal{W}$ but not contained in $\mathbb{C}_{\oboxtimes}$. We denote by $w^{1}$ the first letter of the first non-empty word in $w$. Then $\bar{\Delta}(w) = \Delta_{\prec}(w) + \Delta_{\succ}(w)$ with
\begin{equation}
	\label{eqn:unshuffle}
	\Delta_{\prec}(w) = \sum_{\substack{\alpha,\beta \in \mathcal{W}, \\ w=\nabla_{\mathcal{W}}(\alpha \oboxtimes \beta),\\ w^{1}\in\alpha,\,\alpha \neq w}} \alpha \oboxtimes \beta, \quad \Delta_{\succ}^{+,\mathcal{W}}=\sum_{\substack{\alpha,\beta \in \mathcal{W}, \\ w=\nabla_{\mathcal{W}}(\alpha \oboxtimes \beta),\\ w^{1}\in \beta,\,\beta \neq w}} \alpha \oboxtimes \beta.
\end{equation}
Finally, define $S^{\mathcal{W}}: \mathcal{W} \rightarrow \mathcal{W}$ as the unique horizontal morphism such that:
\begin{equation}
	\label{eqn:antipode}
	S^{\mathcal{W}}(a_{1}\cdots a_{p}) = (-1)^{p}a_{1}\cdots a_{p},\quad a_{1}\cdots a_{p} \in \bar{T}(\mathcal{A}).
\end{equation}

\begin{proposition}
	\label{prop:shuffleword}
	$(\mathcal{W},\Delta^{\mathcal{W}}_{\prec},\Delta^{\mathcal{W}}_{\succ},\nabla_{\mathcal{W}},S^{\mathcal{W}})$ is an unshuffle \raisebox{-1pt}{$\oboxtimes$}$\gtimes$-Hopf algebra.
\end{proposition}
\begin{proof}
We only sketch the proof, the same machinery of cuts and admissible cuts expounded for the gap-insertion operad applies here.
Let $w$ be a word in $\mathcal{W}$ containing at least one non-empty word. By definition, such a word can be written $w = \emptyset^{q} | x | w^{\prime}$, with $x$ a word in $T(\mathcal{A})$ not equal to the empty word. We call a lowerset of $x$ a subset $S$ of letters of $x$. Then a lowerset determines a sequence of words $S^{|} = U_{0} | \ldots |U_{|S|}$, each of the $U_{i}$ being either an empty words or a connected component of the complementary set of $S$ in $x$. We have:
\begin{equation}
	\Delta^{\mathcal{W}}(x) = \sum_{S \subset x} S \oboxtimes S^{|}.
\end{equation}
The unique lowerset of the empty word is the empty word itself and $\emptyset^{|}=\emptyset$. The notion of is readily extended to words on words.
An upperset of $x$ is a sequence $U_{0}|\ldots|U_{s}$ such that each of the $U_{i}$ is either the empty word of a subword of $x$, with the condition that there exists a subword $L \in x$ (a lowerset) of length $s$ such that $x = \nabla_{\mathcal{W}} (L \oboxtimes U_{0}|\ldots|U_{s})$. Notice that the only upperset of the empty word is the empty word itself. The notion of lowerset is then canonically extended to words on words. We denote by $U^{\cdot}$ the lowerset associated with an upperset of $U$.

A triple cut of $w$ is a triplet $(L,M,U)$ such that $L$ is a lowerset of $w$, $U$ is an upperset of $w$, $L^{|} = \nabla M \oboxtimes U$ and $U^{\cdot} = \nabla L \oboxtimes U$. In that case, $M$ is a lowerset of $L^{|}$ and $U=M^{|}$. We denote by Cut$_{2}(w)$ the set of triple cuts of $w$ such that $L,M,U$ are not in $\pmb{\mathbb{C}}_{\oboxtimes}$. The following relations hold:

\begin{align*}
	 & (\Delta^{\mathcal{W}}_{\prec} \oboxtimes \textrm{id}) \circ \Delta^{\mathcal{W}}_{\prec} (w) = \sum_{\substack{(L,M,U)\in \mathrm{Cut}_{2}(w) \\ w^{1} \in L}} L \oboxtimes M^{} \oboxtimes U^{} = \textrm{id} \oboxtimes \Delta^{\mathcal{W}} \circ \Delta^{\mathcal{W}}_{\prec} (w),\\
	 & (\Delta^{\mathcal{W}}_{\succ} \oboxtimes \textrm{id}) \circ \Delta^{\mathcal{W}}_{\prec} (w) = \sum_{\substack{(L,M,U)\in \mathrm{Cut}_{2}(w) \\ w^{1} \in M}} L \oboxtimes M^{} \oboxtimes U^{} = \textrm{id} \oboxtimes \Delta^{\mathcal{W}}_{\prec} \circ \Delta^{\mathcal{W}}_{\succ} (w)\\
	 & (\Delta^{\mathcal{W}} \oboxtimes \textrm{id}) \circ \Delta^{\mathcal{W}}_{\succ} (w)= \sum_{\substack{(L,M,U)\in \mathrm{Cut}_{2}(w)          \\ w^{1} \in U}} L \oboxtimes M^{} \oboxtimes U^{} = \textrm{id} \oboxtimes \Delta^{\mathcal{W}}_{\succ} \circ \Delta^{\mathcal{W}}_{\prec} (w)
\end{align*}
\end{proof}
We set $\mathcal{W}^{+} = \bigoplus_{n\neq m} \mathcal{W}(n,m)$. Proposition \ref{prop:shuffleword} implies that the class of bicollection homomorphisms $\mathrm{Hom}_{\mathrm{Coll}_{2}}(\mathcal{W}^{+}, T_{\gtimes}(\mathrm{Hom}(B)))$ is a shuffle algebra. We set
\begin{equation}
	\overline{\mathrm{Hom}}_{\mathrm{Coll}_{2}}(\mathcal{W}, T_{\gtimes}(\mathrm{Hom}(B))) = \mathbb{C}\eta_{\mathrm{Hom}(B)} \circ \varepsilon^{\mathcal{W}} \oplus \mathrm{Hom}_{\mathrm{Coll}_{2}}(\mathcal{W}^{+}, T_{\gtimes}(\mathrm{Hom}(B)))
\end{equation}
The following equations endow $\overline{\mathrm{Hom}}_{\mathrm{Coll}_{2}}(\mathcal{W}, T_{\gtimes}(\mathrm{Hom}(B)))$ with the structure of an augmented shuffle algebra:
\begin{align*}
	&\eta_{\mathrm{Hom}(B)} \circ \varepsilon^{\mathcal{W}} \prec \alpha = \alpha \succ \eta_{\mathrm{Hom}(B)} \circ \varepsilon^{\mathcal{W}} = 0,\quad
	\eta_{\mathrm{Hom}(B)} \circ \varepsilon^{\mathcal{W}} \succ \alpha = \alpha \prec \eta_{\mathrm{Hom}(B)} \circ \varepsilon^{\mathcal{W}} = \alpha
\end{align*}

\subsection{The splitting map}
We define $Sp : \mathcal{W} \rightarrow T_{\gtimes}(\mathcal{N}\mathcal{C})$ the \emph{splitting map} in our settings, following \cite{ebrahimi2016splitting}. It is an horizontal morphism extending to $W$:
\begin{equation*}
	Sp(a_{1}\cdots a_{p}) = \sum_{\pi \in \mathrm{NC}(p)} \pi \otimes (a_{1} \cdots a_{p}), \quad a_{1}\cdots a_{p} \in \bar{T}(\mathcal{A}),
\end{equation*}
\begin{proposition}
	\label{prop:unshuffleword}
	The horizontal algebra morphism $Sp$ is an unshuffle morphism, which means:
	$$(Sp \oboxtimes Sp) \circ \Delta_{\prec,\succ}^{\mathcal{W}} = \Delta_{\prec,\succ}^{T_{\gtimes}(\mathcal{N}\mathcal{C})} \circ Sp,~\varepsilon^{T_{\gtimes}(\mathcal{N}\mathcal{C})} \circ Sp = \varepsilon^{\mathcal{W}}.$$
\end{proposition}
\begin{proof}
	The arguments exposed in \cite{ebrahimi2016splitting} can be used verbatim to prove the result. Let us prove the statement involving the two coproducts $\Delta^{\mathcal{W}}$ and $\Delta^{T_{\gtimes}(\mathcal{N}\mathcal{C})}$. It is enough to show that
	\begin{equation}
		(Sp\oboxtimes Sp) \circ \Delta^{\mathcal{W}} (a_{1}\cdots a_{n})= \Delta^{T_{\gtimes}(\mathcal{N}\mathcal{C})} (Sp(a_{1}\cdots a_{n})),~a_{1}\cdots a_{n} \in \mathcal{A}^{\otimes n}.
	\end{equation}
	\begin{align*}
		\Delta^{T_{\gtimes}(\mathcal{N}\mathcal{C})} (Sp(a_{1}\cdots a_{n})) & = \sum_{\pi \in \mathrm{NC}(p)} \sum_{\nabla(\alpha \oboxtimes (\beta_{1},\ldots\beta_{|\alpha|})) = \pi} (\alpha\otimes a_{\alpha})\oboxtimes (\beta_{1}\otimes a_{\beta_{1}}) \otimes \cdots \otimes (\beta_{|\alpha|}\otimes a_{\beta_{|\alpha|}})
	\end{align*}
	In the last equation, the second sum runs over non-crossing partitions $\alpha$, and $\beta_{1},\ldots,\beta_{\alpha}$ seen as subsets of $\pi$, with the condition that the operadic composition (in the operad $\mathcal{N}\mathcal{C}$) $\nabla(\alpha \oboxtimes (\beta_{1},\ldots\beta_{|\alpha|}))$ of their \emph{standard representatives} is equal to $\pi$. The notation $a_{\alpha}$ is the word in $T(\mathcal{A})$ obtained from $a$ by concatenation of the linearly order set of letters partitioned by $\alpha$, by convention $a_{\{\emptyset\}} = \emptyset$. In the vein of the proof of the preceding proposition, if $S \subset \llbracket 1,n\rrbracket$ is a (possibly empty) set, we denote by $U_{0},\ldots,U_{|S|}$ the words in $T(\mathcal{A})$ such that $a=\nabla_{\mathcal{W}}(a_{S} \oboxtimes U_{0}\gtimes \cdots \gtimes U_{|S|})$. Then, we see that:
	\begin{align*}
		\Delta^{T_{\gtimes}(\mathcal{N}\mathcal{C})}(Sp(a_{1}\cdots a_{n})) & = \sum_{S\subset\llbracket 1,n\rrbracket}\sum_{\substack{\alpha \in \mathrm{NC}(S), \\  \beta_{1} \in \mathrm{NC}(U_{0}),\ldots,\beta_{|\alpha|} \in \mathrm{NC}(U_{|S|})}} (\alpha\otimes a_{S})\oboxtimes (\beta_{0}\otimes U_{0}) \otimes \cdots \otimes (\beta_{|\alpha|}\otimes U_{|S|}) \\
		                                                                    & = (Sp \oboxtimes Sp)\circ \Delta^{\mathcal{W}}(a_{1}\cdots a_{n}).
	\end{align*}
\end{proof}
An equivalent statement to Proposition \ref{prop:unshuffleword} is that the dual of the splitting map induced a morphism between the two augmented unshuffle algebra of bicollections morphisms on $\mathcal{W}$ and on $T_{\gtimes}(\mathcal{N}\mathcal{C})$ with values in $T_{\gtimes}(\mathrm{Hom}(B))$.
\begin{remarque}
	The map $\Delta^{\mathcal{W}}$ is an horizontal morphism (by definition) but is not a PROS morphism: $\nabla_{T_{\gtimes}(\mathcal{N}\mathcal{C})} \circ (Sp \oboxtimes Sp) \neq Sp \circ \nabla_{\mathcal{W}}$. As a consequence, the splitting morphism $Sp$ is not a $\oboxtimes \gtimes$-Hopf algebra morphism, in particular:
	\begin{equation*}
		(S^{T_{\gtimes}(\mathcal{N}\mathcal{C})} \circ Sp)(w)  = \sum_{\pi \in \mathrm{Int}(p)} (-1)^{\sharp \pi} \otimes w \neq (Sp \circ S^{\mathcal{W}})(w) = \sum_{\pi \in \mathrm{NC}(p)} (-1)^{|w|} \pi \otimes w.
	\end{equation*}
\end{remarque}



\begin{definition}
	An \emph{infinitesimal morphism} $k:\mathcal{W} \rightarrow T_{\gtimes}(\mathrm{Hom}(B))$ is a bicollection map equal to zero on every word in $\mathcal{W}$ except that
	\begin{equation}
		k(\emptyset^{p}|w|\emptyset^{q}) = id^{p} | k(w) | id^{q},\quad w \in T(\mathcal{A}),~ w \neq \emptyset.
	\end{equation}
\end{definition}

The following lemma is a corollary of Proposition \ref{prop:unshuffleword} and the computations of the shuffle exponentials (of infinitesimal morphisms from the gap-insertion PROS to the endomorphism PROS of $B$) of the previous sections.

\begin{lemma}
	\label{lem:computationlefthalfshuffle}
	Let $k : \mathcal{W} \rightarrow T_{\gtimes}(\mathrm{Hom}(B))$ be an infinitesimal morphism satisfying
	\begin{equation*}
		k(a_{1}\otimes \cdots \otimes a_{p}) \circ_{1} k(a^{\prime}_{1}\otimes \cdots \otimes a^{\prime}_{q}) = k(a^{\prime}_{1}\otimes \cdots \otimes a^{\prime}_{q}) \circ_{q+1} k(a_{1}\otimes \cdots \otimes a_{p}),~a_{i},a^{\prime}_{i} \in \mathcal{A}.
	\end{equation*} Then
	the following formulas hold
	\begin{align}
		\label{eqn:exponentialword}
		 & \exp_{\prec}(k)(w)(b_{0} \otimes \cdots \otimes b_{p}) = \sum_{\pi \in \mathrm{NC}(p)} \exp^{T_{\gtimes}(\mathcal{N}\mathcal{C})}_{\prec}(\underline{k})(\pi \otimes w)(b_{0},\ldots,b_{p}),\quad w \in \mathcal{A}^{\otimes p},                    \\
		 & \exp_{\succ}(k)(w)(b_{0} \otimes \cdots \otimes b_{p}) = \sum_{\pi \in \mathrm{Int}(p)} \exp^{T_{\gtimes}(\mathcal{N}\mathcal{C})}_{\succ}(\pi\otimes w)(b_{0},\ldots,b_{p}),\quad w \in \mathcal{A}^{\otimes p} \nonumber
	\end{align}
	with $\underline{k}$ the infinitesimal morphism on $T_{\gtimes}(\mathcal{N}\mathcal{C})$ satisfying $\underline{k}(\pi\otimes w)=\delta_{\pi = 1_{n}}k(w)$.

\end{lemma}

\begin{proposition}[Operator-valued moment-cumulant relations]
	\label{prop:momentcumulantsrelation}
	With the notation introduced so far, let $k: \mathcal{W} \rightarrow T_{\gtimes}(\mathrm{Hom}(B))$ and $b : \mathcal{W} \rightarrow T_{\gtimes}(\mathrm{Hom}(B))$ be the infinitesimal morphisms on $\mathcal{W}$ such that:
	\begin{equation*}
		k(a_{1}\cdots a_{n})(b_{0},\ldots,b_{n}) = \kappa_{n}(b_{0}a_{1},\ldots,a_{n}b_{n}),~b(a_{1}\cdots a_{n})(b_{0},\ldots,b_{n}) = \beta_{n}(b_{0}a_{1},\ldots,a_{n}b_{n}).
	\end{equation*}
Besides denotes by $E$ the horizontal morphism on $\mathcal{W}$ with values in $T_{\gtimes}(\mathrm{Hom}(B))$ such that:
	\begin{equation*}
		E(a_{1}\ldots a_{n})(b_{0},\ldots,b_{n}) = E(b_{0}\cdot_{\mathcal{A}} a_{1} \cdots a_{n}\cdot_{\mathcal{A}}b_{n}).
	\end{equation*}
	Then:
	\begin{equation*}
		E = \eta_{\mathrm{Hom}(B)}\circ \varepsilon^{\mathcal{W}} + k\prec E = \eta_{\mathrm{Hom}(B)}\circ \varepsilon^{\mathcal{W}} + E\succ b.
	\end{equation*}
\end{proposition}
\begin{proof}
	Denote by $\underline{k}$ and $\underline{b}$ the infinitesimal morphisms from $T_{\gtimes}(\mathcal{N}\mathcal{C})$ to  $T_{\gtimes}(\mathrm{Hom}(B))$ defined by:
	\begin{align*}
		 & \underline{k}(\pi\otimes a_{1}\otimes \cdots \otimes a_{n})(b_{0},\ldots,b_{n}) = \delta_{\pi=\mathbf{1}_{n}}\kappa_{n}(b_{0}a_{1},\ldots,a_{n}b_{n}),           \\
		 & \underline{b}(\pi \otimes a_{1}\otimes \ldots \otimes a_{n})(b_{0},\ldots,b_{n}) = \delta_{\pi=\mathbf{1}_{n}}\beta_{n}(a_{1},\ldots,a_{n})(b_{0},\ldots,b_{n}),
	\end{align*}
	where $\kappa_{p}(a_{1},\ldots,a_{n})$ respectively $\beta_{n}(a_{1},\ldots,a_{n})$ are the operator-valued free cumulant respectively boolean cumulants of the random variables $a_{1},\ldots,a_{p}$. Then the maps $k,b : \mathcal{W} \rightarrow T_{\gtimes}(\mathrm{Hom}(B))$ defined by
	\begin{equation}
		k(a_{1}\cdots a_{p}) = (\underline{k} \circ Sp)(a_{1}\cdots a_{p}),~ b (a_{1} \cdots a_{p})= (\underline{b}\circ Sp) (a_{1}\cdots a_{p})
	\end{equation}
	are infinitesimal morphisms on $\mathcal{W}$. Let $K$ and $B$ be the horizontal morphisms from $T_{\gtimes}(\mathcal{N}\mathcal{C})$ to $T_{\gtimes}(\mathrm{Hom}(B))$ solutions of the fixed point equations
	\begin{equation}
		K = \eta_{\mathrm{Hom}(B)} \circ \varepsilon^{T_{\gtimes}(\mathcal{N}\mathcal{C})} + \underline{k} \prec K,~ B = \eta_{\mathrm{Hom}(B)} \circ \varepsilon^{T_{\gtimes}(\mathcal{N}\mathcal{C})} + B \succ \underline{b}.
	\end{equation}
	Owing to Proposition \ref{prop:unshuffleword}, the morphisms $K \circ Sp$ and $E \circ Sp$
	are solutions of the following fixed point equations:
	\begin{equation}
		K \circ Sp = \eta \circ \varepsilon^{\mathcal{W}} + k \prec (K \circ Sp),~
		B \circ Sp  = \eta \circ \varepsilon^{\mathcal{W}} + (B\circ Sp) \succ b.
	\end{equation}
	Now owing to Lemma \ref{lem:computationlefthalfshuffle} and definitions of the free and boolean cumulants, we have
	\begin{equation*}
		(K \circ Sp) (a_{1}\cdots a_{n})(b_{0},\ldots,b_{n}) = (B \circ Sp) (a_{1}\cdots a_{n})(b_{0},\ldots,b_{n}) = E(b_{0}a_{1}\cdot_{\mathcal{A}} \cdots \cdot_{\mathcal{A}}a_{n}b_{n}),~ a_{1},\ldots,a_{n} \in \mathcal{A}.
	\end{equation*}
\end{proof}

\bibliographystyle{plain}
\bibliography{operatorvaluedcumulantmoment}
\nocite{*}
\end{document}